\documentclass[12pt]{article}

\usepackage{amsfonts,amsmath, amssymb}

\usepackage{xcolor}

\usepackage{enumerate}
\usepackage{epsf,epsfig,amsfonts,a4wide}
\usepackage{amsfonts, mathdots}
\usepackage{amsmath,amssymb,amsthm}
\usepackage{url}
\usepackage{amsmath,amssymb,amsthm}

\usepackage{amssymb}
\usepackage{amsthm}
\usepackage{epsf,epsfig,amsfonts,graphicx}
\usepackage{a4wide}
\newtheorem{Theorem}{Theorem}[section]
\newtheorem{Lemma}{Lemma}[section]
\newtheorem{Proposition}{Proposition}[section]
\newtheorem{Corollary}{Corollary}[section]

\newtheorem{Definition}{Definition}[section]
\newtheorem{Ex}{Example}[section]
\newtheorem{Remark}{Remark}[section]

\theoremstyle{remark}

\newcommand{\be}{\begin{equation}}
\newcommand{\ee}{\end{equation}}

\newcommand{\R}{\mathbb{R}}\newcommand{\Id}{\textrm{\rm Id}}

\newcommand{\Image}{\mathrm{Im}\,}
\newcommand{\sfk}{\mathsf{k}}
\newcommand{\sfm}{\mathsf{m}}
\newcommand{\sfl}{\mathsf{l}}

\newcommand{\ddd}{\mathrm{d}}

\newcommand{\pd}[2]{\frac{\partial#1}{\partial#2}}

\newcommand{\dd}{{\mathrm d}\,}

\newcommand{\ii}[0]{\mathrm{i \,}}

\newcommand{\iii}{\mathrm{i \,}}

\newcommand{\tr}{\operatorname{tr}}

\def\Ker{\operatorname{Ker}}
\def\Image{\operatorname{Image}}

\newcommand{\trace}{\operatorname{tr}}

\newcommand{\weg}[1]{}

\title{Nijenhuis Geometry}
\author{Alexey V. Bolsinov\footnote{ School of Mathematics,
 Loughborough University,
 LE11 3TU, UK   and   Faculty of Mechanics and Mathematics, Moscow State University, 119992, Moscow Russia\ \
 \quad {\tt A.Bolsinov@lboro.ac.uk} } \quad
\& \quad  Andrey Yu. Konyaev\footnote{Faculty of Mechanics and Mathematics, Moscow State University, 119992, Moscow Russia
 \ \ \quad {\tt  maodzund@yandex.ru}} \quad \& \quad Vladimir S. Matveev\footnote{
Institut f\"ur Mathematik, Friedrich Schiller Universit\"at Jena,
07737 Jena Germany  \ \ \quad {\tt  vladimir.s.matveev@gmail.com}} 
}  
\date{}

\begin{document}

\maketitle

\begin{abstract}
This work is the first, and main, of a series of papers in progress dedicated to Nienhuis operators, i.e., fields of endomorphisms with vanishing Nijenhuis tensor. It serves as an introduction to Nijenhuis Geometry that should be understood in much wider context than before: from local description at generic points to singularities and global analysis. The goal of the present paper is to introduce terminology, develop new important techniques (e.g., analytic functions of Nijenhuis operators, splitting theorem and linearisation), summarise and generalise basic facts (some of which are already known but we give new self-contained proofs), and more importantly, to demonstrate that the research programme proposed in the paper is realistic by proving a series of new, not at all obvious, results. 
\end{abstract}


\tableofcontents

\section{Introduction}

Our motivation is based on  the following `philosophical' question: 
what is the simplest  local  geometric  condition on a tensor with two indices? 

For bivectors  $P=P^{ij}\partial_{x^i} \wedge \partial_{x^j}$,  such a condition is, possibly,  the Jacobi identity, so that  $P^{ij}$ defines  a Poisson bracket. 
Similarly, for differential 2-forms $\omega=\omega_{ij}dx^i\wedge dx^j$,  the simplest geometric condition is the closedness,   $\ddd\omega=0$,  leading to the notion of a symplectic structure.

For non-degenerate symmetric tensors $g=g_{ij}\ddd x^i\ddd x^j$, (pseudo)-Riemannian metrics,  a  natural candidate is the condition that $g$ has constant curvature or is just flat. There are, however, other more sophisticated properties like being locally symmetric,  conformally flat, Einstein or having holonomy group with special properties.
Since non-degenerate $(2,0)$-tensors  are in one-to-one correspondence with non-degenerate $(0,2)$-tensors,  these conditions can be reformulated for $(2,0)$ symmetric tensors as well\footnote{Of course, for non-degenerate skew-symmetric tensors we can do the same, then the above two conditions become equivalent:  $\omega$ is closed if and only if $P=\omega^{-1}$ satisfies the Jacobi identity.}.  For degenerate symmetric tensors $B=(b^{ij})$,  one may consider integrability condition for the kernel of $B$.  Symmetric tensors $B$ with totally non-integrable distribution defined by $\Ker B$  are known as sub-Riemannian structures.

These examples illustrate the fact that simple assumptions lead to important geometric structures and give rise to fundamental theories.  The above geometric conditions and related structures appeared independently and proved to be useful in different  topics in mathematics and  mathematical physics. Notice that in most cases,  the rich geometric theories that arise from these conditions come from both the local description and the global analysis including, in particular,  topological obstructions for existence of such structures on compact manifolds and the behaviour of singularities. The latter was especially important for Poisson structures and, due to A.\,Weinstein \cite{weinstein}, became the starting point of modern Poisson geometry.

 
The main objects of our paper  are  $(1,1)$-tensors.  As the simplest geometric condition in this case,  it is natural to choose the condition $\mathcal{N}_L=0$, where $\mathcal{N}_L$ is the Nijenhuis tensor of the $(1,1)$-tensor $L= (L^i_j)$. Recall that 
$\mathcal{N}_L$ is the $(1,2)$-tensor (vector valued 2-form) defined by    
$$
 \mathcal{N}_L(v, w) = L^2 [v, w] + [Lv, Lw] - L [Lv, w] - L [v, Lw]
$$
for arbitrary vector fields $v$ and $w$. 

\begin{Definition}\label{def:nijop}
{\rm
By \emph{Nijenhuis operators} we understand $(1,1)$-tensors with vanishing Nijenhuis tensor. A manifold $M$ together with such an operator defined on it is called a {\it Nijenhuis manifold}.   
}\end{Definition}

Our ultimate goal is to answer the following natural  questions, which are motivated by the `philosophical' discussion above:  

\begin{enumerate} 

\item Local description: to what form can one bring a Nijenhuis operator near almost every point by a local coordinate change?  \label{q:1} 

\item Singular points: what does it mean for a point to be generic or singular in the context of Nijenhuis  geometry?
What singularities are non-degenerate? What singularities are stable?  How do Nijenhuis operators behave near non-degenerate and stable singular points?  
 
 \item Global properties: what restrictions on 
a Nijenhuis operator are imposed by compactness of the manifold?  And conversely,  what are topological obstructions to a Nijenhuis manifold carrying a Nijenhuis operator with specific properties (e.g. with no singular points)?    \label{q:3}

\end{enumerate} 

These questions seem to be rather natural on their own and, we believe, deserve to be thoroughly discussed.  Each of us, however, has come to them in his own way via projective geometry and/or theory of bi-Hamiltonian systems.  We would like to give a few comments about that. 

In the theory of bi-Hamiltonian systems, which goes back to the pioneering work of F.\,Magri \cite{magri1},  Nijenhuis operators appear as recursion operators. Their role in this area has been well understood for many years \cite{magri2, kossman, turiel}.  However,  only  recently  has it  been understood that singularities of Lagrangian fibrations related to bi-Hamiltonian systems should apparently be controlled by singularities of the corresponding recursion operator (A.B. and A.\,Izosimov \cite{BolsIzos},  see also \cite{openprob}).

In the theory of integrable systems on Lie algebras,  one often considers algebraic Nijenhuis operators $L: \mathfrak g \to \mathfrak g$ defined on a Lie algebra $(\mathfrak g,  [~,~])$ and satisfying the identity $L^2[x, y] + [Lx, Ly] - L[Lx, y] - L[x, Ly] = 0$  (see e.g. \cite{kossman, pan}).  Such an operator  on $\mathfrak g$ naturally induces a left-invariant operator on the corresponding Lie group $G$ which is automatically Nijenhuis in the sense of Definition \ref{def:nijop}.  Algebraic Nijenhuis operators turned out to be the main tool  used by A.K. \cite{akon} to study Sokolov-Odesskii systems introduced in \cite{sokodess}.

In metric projective geometry, one of the important topics is a description of closed manifolds admitting pairs of geodesically equivalent Riemannian or pseudo-Riemannian metrics. There are strong topological obstructions for such manifolds.  Some of them have been discovered by V.M. \cite{MatveevRigid}.  One of possible approaches to the problem could be based on the fact that such a manifold $M$ always carries a globally defined Nijenhuis operator (see \cite{BoMa2003} by A.B. and V.M. and Section \ref{sect:geod}).

The above are just some examples, from our personal experience, illustrating possible applications of Nijenhuis geometry. Of course, the list can be continued.   We would like to especially refer to the theory of Poisson-Nijenhuis structures  (Y.\,Kosmann-Schwarzbach and F.\,Magri  \cite{kossman}), 
and classification of compatible pairs and multi-dimensional Poisson brackets of hydrodynamic type (E.\,Ferapontov \cite{Fer1},  E.\,Ferapontov,  P.\,Savoldi and P.\,Lorenzoni \cite{Fer2}); more examples can be found in \cite{bogoyavlensky}.

A conceptual explanation why Nijenhuis operators appear  in completely different mathematical setups lies in the theory of natural geometric operations \cite{Kolar}:  the Nijenhuis tensor $\mathcal N_L$ is  one of the simplest geometric objects that can be constructed from an operator $L$.  Moreover, the correspondence $L \mapsto \mathcal N_L$  is 
the only nontrivial geometric operation from ``(1,1)-tensors'' to ``(1,2)-tensors''  that is homogeneous of degree 2  (``trivial'' operations of this type are just suitable algebraic expressions in $L$, $\ddd\tr L$ and $\ddd\tr L^2$), see \cite{Kolar}, \cite{Puninskii}. 
 Therefore, it should not be surprising that compatibility conditions  of a geometric system of partial differential equations  whose coefficients can be written in terms of  $L$ are often related to vanishing of $\mathcal N_L$. That is what happens 
in the above examples where vanishing of the Nijenhuis tensor occurs as a compatibility condition for PDEs defining the corresponding geometric structure.

This work is the first, and main, of a series of papers in progress dedicated to the questions \ref{q:1}--\ref{q:3}. 
It serves as an introduction to Nijenhuis Geometry that, in view of the above agenda, should be understood in much wider context than before:  from local description at good generic points to singularities and global analysis.  The goal of the present paper is to introduce terminology,  develop new important techniques (e.g., analytic functions of Nijenhuis operators, splitting theorem and linearisation), summarise and generalise basic facts (some of which are already known but we give new self-contained proofs), and more importantly, to demonstrate that the proposed research programme is realistic by proving a series of new, not at all obvious, results  (e.g.,  Corollaries \ref{cor:stableblock}, \ref{cor:Matv1} and \ref{cor:Matv2}).

The paper is organised as follows.  We start  with a technical Section \ref{sect:2}
containing the list of useful formulas, basic facts and equivalent definitions for Nijenhuis operators which will be  used later.  One of the key points here is  our treatment of the characteristic polynomial of a Nijenhuis operator,  namely, formulas \eqref{eq10} and \eqref{eq:Lexplicit} leading to a general principle (Corollary \ref{cor:Bols1}) stating that a Nijenhuis operator can be uniquely recovered from the characteristic polynomials if its coefficients are functionally independent almost everywhere.

The most important result of Section \ref{sect:3} is, possibly, the splitting theorem which basically reduces the local study of Nijehhuis operators to the case when $L$ has only one eigenvalue at a given point (Theorems \ref{thm1}, \ref{th:bols3.2}).  This point need not to be generic in any sense!  Another useful technique developed in this section is  based on studying matrix functions $f(L)$ of Nijenhuis operators. This  allows us, in particular, to analyse real Nijenhuis operators $L$ with no real eigenvalues. We show that for such an operator one can canonically construct 
 a complex structure on the manifold with respect to which $L$ automatically becomes holomorphic. 
  
In Section  \ref{sect:4} we discuss local canonical forms for Nijenhuis operators in a neighbourhood of a generic point or of a singular point which is in some sense non-degenerate. We give slightly modified versions of classical diagonalisability theorems going back to Nijenhuis \cite{nij} and Haantjes \cite{haant}  under three different assumptions  ($\mathrm{gl}$-regularity,  algebraic genericity and functional independence of eigenvalues) and also describe canonical forms for Jordan blocks (Theorem \ref{thm:part1:10}). Combining Theorems \ref{th:bols3.2},  \ref{thm:part1:6},  \ref{thm:part1:10} and Remark \ref{rem:canformgeneral} leads us to a local normal form for Nijenhuis operators  $L$ on $M$  (for an open everywhere dense subset $M_\circ \subseteq M$ of algebraically generic points) under the condition that each eigenvalue of $L$ has geometric multiplicity one.

Another important result of Section \ref{sect:4} relates  to the following situation.  Consider a singular point $x_0\in M^n$  at which a Nijenhuis operator $L$ has a certain prescribed algebraic type, for instance, $L$ is a single Jordan block.  Then we may think of $L(x)$ as an $n$-parametric deformation of $L(x_0)$  which, in general, immediately changes its algebraic type.  What is a typical behaviour of Nijenhuis operators at such points?  We show that in the case of a single Jordan block,  Nijenhuis deformations  exactly coincide with versal deformations in the sense of V.\,Arnol'd \cite{Arnold} (Theorem \ref{thm:part1:7}) and, therefore, are stable  (Corollary \ref{cor:stableblock}).  For an arbitrary algebraic type,  a conjectural answer is given in Theorem  \ref{thm:part1:8} that describes one of possible scenarios for Nijenhuis deformations in the most general case.  However,  the stability of such a scenario is an open question.

Section \ref{sect:5} is devoted to  linearisation of Nijenhuis operators at singular points $p\in M$ of scalar type, i.e., such that $L(p) = c\cdot \Id$.  
This procedure is similar to the linearisation of Poisson structures \cite{weinstein} and leads to the so-called left-symmetric algebras (LSA).  We discuss here the linearisation problem (under what conditions a Nijenhuis operator $L$ is equivalent to its linear part $L_{\mathrm{lin}}$?)  directly related to Question 2 stated above.  The main result of this section is the proof of stability for singular points of diagonal type in the real analytic case.  This means, in other words, that a diagonal LSA is non-degenerate.

Finally, in Section 6, we prove several global results related to Question 3.  Namely we show that a Nijenhuis operator $L$ on a compact manifold cannot have non-constant complex eigenvalues (Theorem \ref{appl1}) and, moreover, if $M = S^4$, then the eigenvalues of $L$ have to be real (Corollary \ref{cor:Matv2}).  Another unexpected observation is that compactness prevents existence of differentially non-degenerate singular points (Corollary \ref{cor:Matv1}).  Two other applications are related to the theory of bi-Hamiltonian systems and projective geometry. We study the behaviour (canonical forms) of Poisson-Nijenhuis structures and geodesically equivalent metrics at singular points which, to the best of our knowledge, have never discussed before.

{\bf Acknowledgements.}   The work of A.B. and A.K. was supported by the Russian Science Foundation (project No. 17-11-01303).  Visits of A.B. and A.K. to the University of Jena were supported by DFG (via GK 1523) and DAAD (via Ostpartnerschaft programm), A.B. and A.K. appreciate hospitality and inspiring research atmosphere of the Institute of Mathematics in Jena.  The authors are grateful to A.\,Cap, E.\,Ferapontov,  A.\,Panasyuk, V.\,Novikov,  A.\,Prendergast-Smith,  I.\,Zakharevich, E.\,Vinberg for valuable discussions and comments.


\section{Definitions and basic properties}\label{sect:2}

\subsection{Equivalent definitions of the Nijenhuis tensor}

The original definition of the Nijenhuis tensor was introduced by Albert Nijenhuis in 1951 \cite[formula (3.1)]{nij} as instrument to study the following problem. Consider an operator field $L(x)$ that is diagonalisable at every point $x \in M$.   Does there exist a local coordinate system $x_1,\dots, x_n$ in which $L(x)$ is diagonal, i.e., takes the form $L(x) = \sum_{i} \lambda_i(x)\, \partial_{x_i} \otimes \ddd x_i$?

Let $L$ be a $(1,1)$-tensor field  (operator) on a smooth manifold $M$. The Nijenhuis tensor $\mathcal N_L$ of the operator $L$  is a $(1,2)$-tensor that can be defined in several equivalent ways discussed below.   The first definition is standard and interpret $\mathcal N_l$ as  as a vector-valued 2-form. 

\begin{Definition}
\label{def:Nij1}
{\rm For a pair of vector fields $\xi$ and $\eta$ the Nijenhuis  tensor ${{\mathcal N}_L}$ is defined by the following formula:
$$
{{\mathcal N}_L}(\xi, \eta) = L^2 [\xi, \eta] + [L\xi, L\eta] - L [L\xi, \eta] - L [\xi, L\eta].
$$
}\end{Definition}

If $M$ is equipped with a symmetric connection $\nabla$, then $[\xi, \eta] = \nabla_\xi  \eta - \nabla_\eta \xi$ and the definition can be formulated as follows.

\begin{Definition} {\rm For a pair of arbitrary vector fields $\xi, \eta$ the Nijenhuis tensor ${{\mathcal N}_L}$ is defined by the formula:
$$
{{\mathcal N}_L}(\xi, \eta) = \big(L \nabla_{\xi}L - \nabla_{L\xi}L\big) \eta - \big(L \nabla_{\eta}L - \nabla_{L\eta}L\big) \xi.
$$
}\end{Definition}

\begin{Definition} 
\label{def:Nij2}
{\rm In local coordinates $x^1,\dots, x^n$, the components $(\mathcal N_L)^i_{jk}$ of ${{\mathcal N}_L}$ are defined by the following formula:
$$
(\mathcal N_L)^i_{jk} = L^l_j \pd{L^i_k}{x^l} - L^l_k \pd{L^i_j}{x^l} - L^i_l \pd{L^l_k}{x^j} + L^i_l \pd{L^l_j}{x^k},
$$
where $L^i_j$ denote the components of $L$.

}\end{Definition}

Since $\mathcal N_L$ is a tensor of type $(1,2)$  we can also interpret it  as a linear map from $T M$ to $\mathrm{End}(T M)$. On the other hand, taking into account that $(\mathcal N_L)^i_{jk}$ is skew-symmetric w.r.t. the lower indices $j$ and $k$,  we can treat $\mathcal N_L$ as a linear map from the space $\Omega^1 (M)$ of  differential 1-forms to $\Omega^2 (M)$, the space of differential 2-forms. This leads to the following two definitions.

\begin{Definition} 
\label{def:Nij3}
{\rm 
If we consider ${{\mathcal N}_L}$ as a map from ``vector fields''  to ``endomorphisms'', then 
$$
{{\mathcal N}_L} : \  \xi \mapsto   L\mathcal L_{\xi}L - \mathcal L_{L\xi}L,
$$
where $\mathcal L_\xi$ denotes the Lie derivative along $\xi$.
}\end{Definition}

\begin{Definition}
\label{def:Nij4}
{\rm 
If we consider ${{\mathcal N}_L}$ as a map from ``1-forms''  to ``2-forms'', then 
$$
{{\mathcal N}_L} : \   \alpha \mapsto \beta,
$$
where
\begin{equation}
\label{eq:NijForForms}
\beta( \cdot\, , \cdot )= \ddd({L^*}^2\alpha) (\cdot\, ,\cdot) + \ddd\alpha (L\cdot\, ,L\cdot) -
\ddd(L^*\alpha) (L\cdot\, , \cdot) - \ddd(L^*\alpha) (\cdot\, , L\cdot),
\end{equation}
$L^*:\Omega^1(M)\to \Omega^1(M)$ denotes the dual operator for $L$
and $\ddd: \Omega^1(M) \to \Omega^2(M)$ is the standard exterior derivative.
}
\end{Definition}

To develop a kind of perturbation theory for Nijenhuis operators in Section \ref{sect:5.2}, we will need another fundamental operation on $(1,1)$-tensor fields called the Fr\"olicher--Nijenhuis bracket.

\begin{Definition}\label{def:FNbracket}
{\rm
Let $L_1$ and $L_2$ be $(1,1)$-tensor fields on $M$. Then the 
{\it Fr\"olicher--Nijenhuis bracket} of $L_1$ and $L_2$ is defined as
$$
[L_1, L_2]_{\mathrm{FN}} = \mathcal N_{L_1+L_2} - \mathcal N_{L_1} - \mathcal N_{L_2}
$$
or, equivalently,
$$
\begin{aligned}[]
  [L_1, L_2]_{\mathrm{FN}}   (\xi, \eta) &=
(L_1L_2 + L_2L_1) [\xi, \eta] + [L_1\xi, L_2\eta] +  [L_2\xi, L_2\eta] \\ &- L_1[L_2\xi, \eta]  - L_2 [L_1\xi, \eta]- L_2 [\xi, L_1\eta] - L_1 [\xi, L_2\eta],
\end{aligned}
$$
for any vector fields $\xi$ and $\eta$.
}\end{Definition}


\subsection{Algebraically generic and singular points.  Stability}

As explained in Introduction, one of our ultimate goals is to study the behaviour of Nijenhuis operators at singular points. In this section we introduce several types of singularities for Nijenhuis operators and discuss some natural regularity conditions.

Let $L$ be a Nijenhuis operator on a smooth manifold $M$.  At each point $p\in M$,  we can define the {\it algebraic type}  (or {\it Segre characteristic}, see e.g. \cite{Fraser}) of  $L(p):  T_pM \to T_pM$  as the structure of its Jordan canonical form which is characterised by the sizes of Jordan blocks related to each eigenvalue $\lambda_i$ of $L(p)$ (the specific values of $\lambda_i$'s are not important here).

\begin{Definition}\label{def:alggen}
{\rm
A point $p\in M$ is called {\it algebraically generic}, if the algebraic type of $L$ does not change in some neighbourhood $U(p)\subset M$.  In such a situation, we will also say that $L$ is algebraically generic at $p\in M$.
}\end{Definition}

It follows immediately from the continuity of $L$ that algebraically generic points form an open everywhere dense subset $M_\circ\subseteq M$.

\begin{Definition} \label{def:singpoint}
{\rm
 A point  $p\in M$ is called {\it singular}, if it is not algebraically generic.   
}\end{Definition}

Let us emphasise that in the context of our paper,  ``singular'' is a geometric property and refers to a bifurcation of the algebraic type of $L$  at a given point $p\in M$.   In algebra,  the property of ``being singular or regular''  is usually understood in the context of the representation theory.

\begin{Definition}
\label{def:algregular}
{\rm
An operator  $L(p)$   (and the corresponding point $p\in M$) is called  {\it $\mathrm{gl}$-regular}, if its $GL(n)$-orbit $\mathcal O(L(p))=\{ XL(p)X^{-1}~|~ X\in GL(n)\}$ has maximal dimension, namely, $\dim O(L(p)) = n^2- n $.  Equivalently,  this means that the geometric multiplicity of each eigenvalue  $\lambda_i$ of $L(p)$ (i.e., the number of Jordan $\lambda_i$-blocks in the canonical Jordan decomposition over $\mathbb C$) equals one. }\end{Definition}

One more notion is related to analytic properties of the coefficients of the characteristic polynomial $\chi_{L(x)}(t) = \det \bigl(t\cdot\Id - L(x)\bigr)= \sum_{k=0}^n t^k \sigma_{n-k}(x)$ of $L$.

\begin{Definition}
\label{def:nondeg}
{\rm
A point $p\in M$ is called  {\it differentially non-degenerate}, if the differentials $\ddd\sigma_1(x),\dots,\ddd\sigma_n(x)$ of the coefficients of the characteristic polynomial $\chi_{L(x)}(t)$  are linearly independent at this point.  
}\end{Definition}

The notions of algebraic genericity, algebraic regularity and differential non-degeneracy are related  to each other.  For example, differential non-degeneracy automatically implies algebraic regularity.  On the other hand, 
if $L(p)$ is $\mathrm{gl}$-regular and, in addition,  diagonalisable, then $p$ is algebraically generic (i.e., non-singular).   Conversely, if we assume that  $L(p)$ is not $\mathrm{gl}$-regular, we may expect that $p$ is singular in the sense of Definition \ref{def:singpoint}.  This will definitely be the case if $L$ is real analytic and there are points $x\in M$ at which $L(x)$ is $\mathrm{gl}$-regular.   

Notice that in the context of Definition \ref{def:algregular},  the `most non-regular'  operators are those of the form $L = \lambda\cdot \Id$, $\lambda\in \R$, as in this case $\dim \mathcal O(L) = 0$, i.e., $L$ is a fixed point w.r.t. the adjoint action.    Those points where $L(p)$ becomes scalar play a special role in Nijenhuis geometry  (see Section \ref{sect:4}) and we distinguish this type of points by introducing

\begin{Definition}
\label{def:scalar}
{\rm
A point $p\in M$ is called {\it of scalar type}, if $L(p) = \lambda\cdot \Id$, $\lambda\in \R$.
}\end{Definition}  

In the real analytic case,  scalar type points are automatically singular unless  $L$ is scalar on the whole manifold $M$, i.e.,  $L(x) = \lambda(x) \cdot \Id$  for a certain global function $\lambda : M \to \R$.

As we shall see below,  some of singular points are stable under small perturbations of $L$ in the class of Nijenhuis operators.   Studying and classifying such singularities is one of the most important open problems in Nijenhuis geometry.

\begin{Definition}\label{def:structstable}
{\rm
A singular point $p\in M$ is called ($C^k$-) {\it stable}, if for any perturbation
$$
L(x)  \quad \mapsto  \quad  \widetilde{L}(x) = L(x) + R_k(x)
$$
such that $\widetilde{L}(x)$ is Nijenhuis and $R_k(x)$ has zero of order $k$ at the point $p\in M$,  there exists a local diffeomorphism $\phi :  U(p) \to \widetilde U(p)$, $\phi(p)=p$, that transforms $L(x)$ to $\widetilde L(x)$.
}\end{Definition}

Notice that the above stability property makes sense not only for singular but also for algebraically generic points: some of them are stable, some are not.


\subsection{Invariant formulas}

Many formulas discussed below have already appeared, in this or that form,  in the literature (see e.g. \cite{bogoyavlensky, BoMa2003, BoMa2011, gelzak3, magri2,  turiel}). We give all the proofs to make our paper self-contained and emphasise those ideas which are important for further use.

\begin{Proposition}
\label{prop:inv_form1}  Let $L$ be a Nijenhuis operator, then for any polynomial $p(\cdot)$ with constant coefficients, the operator $p(L)$ is also Nijenhuis. In other words, ${{\mathcal N}_L}=0$ implies ${\mathcal N}_{p(L)}=0$. We also have:
\begin{equation}
\label{eq:trL^k}
\ddd (\tr L^k) = k (L^*)^{k-1} \ddd\tr L,
\end{equation} 
and more generally,
\begin{equation}
\label{eq:trf}
\ddd \bigl(\tr p(L)\bigr) = p'(L)^* \ddd\tr L,
\end{equation}
where  $L^*:T^*_qM \to T^*_qM$ denotes the operator dual  to $L$ and $p'(\cdot)$ is the derivative of $p(\cdot)$.

\end{Proposition}

\begin{proof} We use the condition ${{\mathcal N}_L}=0$ in the form
$\mathcal L_{L\xi} L = L\mathcal L_\xi L$ (see Definition \ref{def:Nij3}).  This identity  implies
\begin{equation}
\label{eq:forProp2.4}
\mathcal L_{L^n \xi} L = \mathcal L_{L(L^{n-1} \xi)} L = L\mathcal
L_{L^{n-1} \xi} L =L\mathcal L_{L(L^{n-2} \xi)} L=L^2\mathcal
L_{L^{n-2} \xi} L = \dots = L^n\mathcal L_\xi L
\end{equation}
and, therefore, by linearity
$$
\mathcal L_{p(L)\xi }L = p(L)\mathcal L_\xi L.
$$
Thus, we have
$$
(\mathcal L_{p(L)\xi} - p(L)\mathcal L_\xi)L=0
$$
Now consider the expression $\mathcal D=\mathcal L_{p(L)\xi} - p(L)\mathcal
L_\xi$ as a ``first order differential operator'' which satisfies the obvious 
property $\mathcal D(L^n)= \mathcal D(L^{n-1}) \, L + L^{n-1} \mathcal D(L)$. Hence  $\mathcal D(L)=0$
immediately implies $\mathcal D(p(L))=0$,~i.e.,
$$
\bigl(\mathcal L_{p(L)\xi} - p(L)\mathcal L_\xi\bigr)p(L)=0,
$$
which is exactly the desired condition ${\mathcal N}_{p(L)}=0$.

To prove \eqref{eq:trL^k} we use \eqref{eq:forProp2.4}  (with $n$ replaced with $k-1$). We have
$$
{\cal L}_{L^{k-1}\xi} \tr L = \tr \left( {\cal L}_{L^{k-1}\xi}  L \right) = \tr \left( L^{k-1} {\cal L}_\xi L \right) = \frac{1}{k}\tr {\cal L}_\xi L^k =
 \frac{1}{k}{\cal L}_\xi \tr L^k,
$$
which is equivalent to \eqref{eq:trL^k}. \end{proof}

\begin{Proposition} \label{prop:00}
Let $L$ be a Nijenhuis operator. 
Then for any vector field $\xi$ we have
$$
{\cal L}_{L\xi}  (\det L)  =
\det L \cdot {\cal L}_\xi  \tr L
$$
or, equivalently,    
\begin{equation}
\label{eq:lndetL}
L^* \ddd  (\det L)=\det L \cdot \ddd \tr L.
\end{equation}
More generally,  the differential of the characteristic polynomial $\chi(t)=\det (t\cdot \Id-L)$ (viewed as a smooth function on $M$ with $t$ as a formal parameter) satisfies the following relation:
\begin{equation}
\label{eq10}
L^* \bigl(\ddd \chi(t) \bigr) - t \cdot \ddd \chi(t)=\chi(t) \cdot \ddd\tr L.
\end{equation}

\end{Proposition}

\begin{proof} We first notice that for any operator $L$ (not necessarily Nijenhuis)
and vector field $\eta$ the following identity   holds:
\begin{equation}
\label{eq:forProp2.5}
{\cal L}_\eta  \det L =\tr \left(\widehat L \, {\cal L}_\eta L\right),
\end{equation}
where  $\widehat L$ denotes the co-matrix of $L$. Now suppose that $L$ is a Nijenhuis operator and therefore ${\cal L}_{L\xi} L = L{\cal L}_\xi L$  
for any vector field $\xi$ (see Definition \ref{def:Nij3}).

Replacing $\eta$ with $L\xi$ in \eqref{eq:forProp2.5} and  using the fact that $\widehat L \, L= \det L \cdot \Id$, we get:
$$
{\cal L}_{L\xi}  \det L =\tr \left(\widehat L \, {\cal L}_{L\xi} L\right) = \tr \left(\widehat L \, L\, {\cal L}_\xi L\right) = \det L \cdot \tr {\cal L}_\xi L=
\det L \cdot {\cal L}_\xi  \tr L,
$$
as stated.

Finally,  formula \eqref{eq10} is obtained from \eqref{eq:lndetL} by replacing $L$ with the Nijenhuis operator $t\cdot \Id-L$ (here we think of $t$ as
a constant).  \end{proof}


\begin{Corollary}
\label{cor:Bols6}
Let $\sigma_1, \dots, \sigma_n$ be the coefficients of the characteristic polynomial 
$$
\chi (t)= \det (t\cdot \Id - L) = t^n + \sigma_1 \, t^{n-1} + \sigma_2 \, t^{n-2} + \dots + \sigma_{n-1} \, t + \sigma_n 
$$
of a Nijenhuis operator $L$.   Then in any local coordinate system $x_1,\dots, x_n$ the following matrix relation hold:
\begin{equation}
\label{eq:Lexplicit}
J(x) \, L (x)= S_\chi (x)  \, J(x),  \quad\mbox{where }   S_\chi (x) = \begin{pmatrix}
-\sigma_1(x)  &  1  &    &    & \\ 
-\sigma_2(x) & 0 & 1&  & \\  
\vdots   &  \vdots &\ddots &\ddots &  \\
-\sigma_{n-1}(x)  & 0 &\dots & 0 &1 \\
-\sigma_n(x) & 0 & \dots & 0 & 0
\end{pmatrix}
\end{equation}
and $J(x)$ is the Jacobi matrix of the collection of functions $\sigma_1,\dots, \sigma_n$ w.r.t. the variables  $x_1,\dots, x_n$. 
\end{Corollary}

\begin{proof}
This matrix relation is equivalent to and can be easily obtained from  \eqref{eq10},  if we rewrite it as $L^* \bigl(\ddd \chi(t) \bigr) = - \chi(t) \cdot \ddd\sigma_1 + t \cdot \ddd \chi(t)$ and equate the terms of the same power in $t$. \end{proof}

Formula \eqref{eq:Lexplicit} implies the following important fact, which can be considered as a kind of uniqueness theorem for Nijenhuis opeartors.

\begin{Corollary}
\label{cor:Bols1}
Assume that the coefficients $\sigma_1, \dots, \sigma_n$ of the characteristic polynomial of a Nijenhuis operator $L$ are functionally independent at a point $p\in M$ (that is, their differentials $\ddd\sigma_1(p), \dots, \ddd\sigma_1(p)$ are linearly independent).  Then $L$ can be uniquely reconstructed from $\sigma_1,\dots, \sigma_n$ in a neighbourhood of $p$.

More globally, if we have two Nijenhuis operators $L_1$ and $L_2$ whose characteristic polynomials on $M$ coincide and  the coefficients of these polynomials  are functionally independent almost everywhere on $M$, then $L_1=L_2$.    
\end{Corollary}

\begin{proof}
Formula \eqref{eq:Lexplicit} allows us, in fact, to get an explicit expression for $L$ in any local coordinate system $x_1,\dots, x_n$ at those points where $\sigma_1,\dots, \sigma_n$ are functionally independent. Namely,  
$L (x)= J^{-1}(x) \, S_\chi (x)\,  J(x)$,
which implies the statement.   \end{proof}

Notice that there are no restrictions on $\sigma_1$, \dots, $\sigma_n$ as soon as they are independent, that is, any collection of independent functions on $M$ can be taken as coefficients of $\chi_L(t)$  for some Nijenhuis operator $L$.  However at those points where the differentials  of $\sigma_1$, \dots, $\sigma_n$ become linearly dependent,  these functions must satisfy rather non-trivial restrictions.

Another straightforward implications of the above discussion is the following reformulation of the definition of a Nijenhuis operator in the special case when $\sigma_1,\dots, \sigma_n$  (or, equivalently, the functions $\tr L, \dots, \tr L^n$)  are independent almost everywhere on $M$.
\begin{Corollary}
\label{cor:Bols11}  Let $L$ be a differentially-nondegenerate operator almost everywhere on $M$, then   the following conditions are equivalent:
\begin{itemize}
\item $L$ is Nijenhuis;
\item relation \eqref{eq:trL^k} holds, i.e., $\ddd (\tr L^k) = k (L^*)^{k-1} \ddd\tr L$ for $k=2,\dots, n+1$;
\item relation \eqref{eq10} holds, i.e., $L^* \bigl(\ddd \chi(t) \bigr) - t \cdot \ddd \chi(t)=\chi(t) \cdot \ddd\tr L$.
\end{itemize}

\end{Corollary}


The next statements describe some properties of the eigenvalues of Nijenhuis operators.

\begin{Proposition}  
\label{prop:dlambda}
Let $L$ be a Nijenhuis operator  and $\lambda(x)$ be a smooth function satisfying the condition $\det (L-\lambda(x) \cdot \Id) \equiv 0$. In other words,  $\lambda(x)$ is a smooth eigenvalue of $L$. Then the differential  $\ddd \lambda$ satisfies the following relation:
\begin{equation}
\label{eq:dlambda}
(L - \lambda (x) )^* \ddd \lambda (x) = 0. 
\end{equation}
\end{Proposition}

\begin{proof} Without loss of generality we may prove this relation at a point $p\in M$ which is spectrally generic in the sense that locally, in a certain neighbourhood of $p$, all eigenvalues of $L$ have constant algebraic multiplicities.  Since the set  of such points is everywhere dense,  \eqref{eq:dlambda} will follow by continuity.  We will also assume that $\ddd \lambda (p)\ne 0$, otherwise the relation is trivial.  One more additional assumption is that $\lambda$ vanishes at  $p\in M$  (otherwise we may replace $L$ with $L -\lambda(p) \cdot \Id$).

To simplify our computations we choose local coordinates $x_1,\dots, x_n$ in such a way that  $\lambda (x) = x_n$.  Notice that under above assumptions, the determinant $\det L$ can be written as $\det L = x_n^k f(x)$ where $f(p)\ne 0$ and $k$ is the multiplicity of $\lambda$.  We need to show that $L^* \ddd x_n=0$.  

We know from \eqref{eq:lndetL} that
$$
L^* (\ddd \det L)=\det L \cdot \ddd\tr L .
$$
Substituting $\det L = x_n^k f(x)$ and dividing  by $x_{n} ^{k-1}$ we obtain:
$$
L^* \bigl( k \, f(x) \ddd x_n  +  x_{n}  \ddd f(x)\bigr) = x_{n} f(x) \cdot \ddd \tr R
$$
Finally taking into account that $\lambda(p)=x_n(p)=0$,  we see that at this point $p\in M$:
$$
L^* \bigl(k  f(p) \ddd x_n \bigr)  = 0
$$
and, since $f(p)\ne 0$,  the statement follows.   \end{proof}

\begin{Corollary}
Let $\xi=\xi(x)$ be a smooth eigenvector field of a Nijenhuis operator $L$, that is $(L - \lambda \cdot \Id) \xi(x) =0$ for a smooth function $\lambda=\lambda(x)$. Then
we have 
$$
(L - \lambda\cdot \Id) \mathcal L_\xi L =0.
$$
\end{Corollary}

\begin{proof} We use the identity $L\mathcal L_\xi L - \mathcal L_{L\xi} L = 0$. Since $\xi$ is an eigenvector field, we get
$$
0 = L\mathcal L_\xi L - \mathcal L_{L\xi} L = L\mathcal L_\xi L - \mathcal L_{\lambda\xi} L=
L\mathcal L_\xi L - \lambda \mathcal L_{\xi} L - \bigl( L^*(\ddd\lambda) \otimes \xi - \ddd\lambda \otimes L\xi\bigr) =
$$
$$
(L - \lambda\cdot\Id ) \mathcal L_\xi L  - (L - \lambda\cdot \Id)^* \ddd\lambda \otimes \xi.
$$
In view of \eqref{eq:dlambda}, we get $(L - \lambda\cdot\Id ) \mathcal L_\xi L =0$ as required. \end{proof}


\subsection{$L$-invariant foliations, restriction and quotient}\label{sect:quotient}

Let $L$ be an  operator (not necessarily Nijenhuis) on a manifold $M$ and $\mathcal F$ be an $L$-invariant foliation (that is, at each point $p\in M$, the tangent space $T_p \mathcal F$ is $L$-invaraint).  In any local coordinate system $x^1,\dots, x^\sfk$, $y^1,\dots, y^\sfm$ adapted to the foliation $\mathcal F$  (i.e. such that  the vectors  $\partial_{x^1}, \dots, \partial_{x^\sfk}$ form a basis of $T\mathcal F$),  the matrix of $L$ takes the form
$$
\begin{pmatrix}
L_1(x,y) & M(x,y) \\
0 & L_2(x,y)
\end{pmatrix}.
$$  

Here  $x = (x^1,\dots, x^\sfk)$ are treated as coordinates on leaves and $y=(y^1,\dots , y^\sfm)$ as coordinates on the space of leaves $M/\mathcal F$.

From the viewpoint of Linear Algebra, at each fixed point $p=(x,y)$ we can define two natural operators,  the restriction of $L$ to $T_{p}\mathcal F$ and the quotient operator acting on $T_{p}M/T_{p}\mathcal F$ which correspond to the diagonal blocks $L_1$ and $L_2$ respectively.  

The restriction $L|_{\mathcal F} \simeq L_1(x,y)$ can be considered as an operator on a particular leaf of  $\mathcal F$ with $y$ being the parameter defining this leaf.  In particular, $L$ can be naturally restricted to every leaf.  

However, if we want to define the quotient operator $\widetilde L \simeq L_2$ as an operator on the (local) quotient space $M/ \mathcal F$,  we need an additional condition, namely, $L_2$-block  should not depend on $x$.   Let us show that this condition does not depend on the choice of an adapted coordinate system and can be defined in an invariant geometric way. 

We will say that a vector field  $\eta$ is $\mathcal F$-{\it preserving}, if the flow of $\eta$ preserves the foliation $\mathcal F$.  Equivalently, for each $\xi\in T\mathcal F$ we have $[\eta,\xi]\in T\mathcal F$, or shortly,
$[\eta,  T\mathcal F]  \subset T\mathcal F$.

\begin{Proposition}
\label{prop:quotient1}
The quotient operator $\widetilde L$ is well defined on $M/\mathcal F$  (equivalently, the diagonal block $L_2$ does not depend on $x$) if and only if for any $\mathcal F$-preserving vector field $\eta$,  the vector field $L\eta$ is also $\mathcal F$-preserving.

If $L$ is a Nijenhuis operator, then its restriction $L|_{\mathcal F}$ onto any leaf of $\mathcal F$  and the quotient operator $\widetilde L$ on $M/\mathcal F$ (provided this operator is well defined in the above sense) are both Nijenhuis.
\end{Proposition}

\begin{proof}
The first statement is related to arbitrary operators (not necessarily Nijenhuis) and is straightforward.  

The fact that the restriction of a Nijenhuis operator $L$ to any $L$-invariant submanifold   (in particular, to a leaf of an $L$-invariant foliation) is still Nijenhuis follows immediately from Definition \ref{def:Nij1}.  

We treat the (local) space of leaves $M/\mathcal F$ as the quotient space with the natural projection $\pi : M \to M/\mathcal F$. It is easy to see that any vector field on the quotient $M/\mathcal F$ can be lifted to $M$ up to an $\mathcal F$-preserving vector field $\eta$  and this lift is defined up to adding an arbitrary vector field $\xi$ tangent to $\mathcal F$.

The proof can be easily done by using an adapted coordinate system. Consider the Nijenhuis condition for $\partial_{y_\alpha}$ and $\partial_{y_\beta}$  (notice that $y_1,\dots,y_\sfm$ form a natural coordinate system on the quotient $M/\mathcal F$):
$$
\mathcal N_L ( \partial_{y_\alpha}, \partial_{y_\beta}) =
L^2 [\partial_{y_\alpha},\partial_{y_\beta}] - L[L\partial_{y_\alpha},\partial_{y_\beta}] - L[\partial_{y_\alpha},L\partial_{y_\beta}] + [L\partial_{y_\alpha},L\partial_{y_\beta}]=
$$
$$
 - L[L\partial_{y_\alpha},\partial_{y_\beta}] - L[\partial_{y_\alpha},L\partial_{y_\beta}] + [L\partial_{y_\alpha},L\partial_{y_\beta}]=
$$
$$
 - L[M_\alpha^i (x,y) \partial_{x_i} + (L_2)_\alpha^\gamma (y) \partial_{y_\gamma},\partial_{y_\beta}] - L[\partial_{y_\alpha},
 M_\beta^j (x,y) \partial_{x_j} + (L_2)_\beta^\nu (y) \partial_{y_\nu}] + 
 $$
 $$
 +[M_\alpha^i (x,y) \partial_{x_i} + (L_2)_\alpha^\gamma (y) \partial_{y_\gamma},M_\beta^j (x,y) \partial_{x_j} + (L_2)_\beta^\nu (y) \partial_{y_\nu}].
$$
Here we assume summation over repeated indices $i, j, \gamma$ and $\nu$.
We can now consider the (pointwise) projection of this identity to the quotient space $T_pM/T_p\mathcal F$, that is, we consider it modulo $\mathrm{span}(\partial_{x_1},\dots, \partial_{x_{\sfk}})$.
 Using the fact that $\mathrm{span}(\partial_{x_1},\dots, \partial_{x_{\sfk}})$ is $L$-invariant and the components of $L_2$ do not depend on $x$, we get
 $$
  (L_2)_\gamma^\nu \frac{\partial (L_2)_\alpha^\gamma}{\partial_{y_\beta}}
  - (L_2)_\gamma^\nu    \frac{\partial (L_2)_\beta^\gamma} {\partial_{y_\alpha}}     + 
  (L_2)_\alpha^\gamma     \frac{\partial (L_2)_\beta^\nu}{ \partial_{y_\gamma}} 
-  (L_2)_\beta^\gamma  \frac{\partial (L_2)_\alpha^\nu} {\partial_{y_\gamma}}     = 0
$$
which is the Nijenhuis condition for the operator given by the matrix $L_2$ in coordinates $y_1,\dots, y_{\sfm}$, i.e., the quotient operator $\widetilde L$.
\end{proof}

With each Nijenhuis operator $L$ we can assign some natural invariant foliations  (integrable distributions).  

\begin{Proposition} 
\label{prop:image}
Let $\lambda(x)$ be a smooth eigenvalue of a Nijenhuis operator $L$ with (locally) constant geometric multiplicity.  
Then the distribution $\Image \bigl(L-\lambda(x)\cdot\Id\bigr)$ is smooth and integrable.    In particular,  if $L$ is degenerate and of (locally) constant rank, then the distribution $\Image L$ is smooth and integrable.
\end{Proposition}

\begin{proof} Smoothness of the distribution $\Image (L-\lambda(x)\cdot\Id)$ is a general fact for all (smooth) operators $L$, not necessarily Nijenhuis.
The integrability of $\Image L$ is straightforward. Indeed,   
$$
L^2[u,v] + [Lu,Lv] - L[Lu,v]- L[u,Lv]=0,
$$
which implies
$$
[Lu,Lv]= L\bigl([Lu,v] + [u,Lv] - L [u,v]\bigr)  \in \Image L.
$$
Since $Lu$ and $Lv$ are arbitrary vector fields from the distribution
$\Image L$, the integrability  follows from the Frobenius theorem.

Next, denote $L_\lambda=L-\lambda (x) \cdot \Id$  and
compute the Lie bracket of  the vector fields $L_\lambda u$ and
$L_\lambda v$ for arbitrary $u$ and $v$:
$$
[L_\lambda u, L_\lambda v ] =
[Lu,Lv]+[\lambda u, \lambda v] - [Lu,\lambda v] - [\lambda u, Lv] =
$$
$$
L[Lu,v] + L[u,Lv] - L^2[u,v] + \lambda^2[u,v] +
\lambda \ddd\lambda(u) v \, - 
$$
$$
\lambda \ddd\lambda (v) u -
\lambda [Lu,v] - \ddd\lambda(Lu) v - \lambda [u,Lv] + \ddd\lambda(Lv) u =
$$
$$
L_\lambda[Lu,v] + L_\lambda [u,Lv] -
L_\lambda(L+\lambda\Id)[u,v] - \ddd\lambda (L_\lambda u) v +
\ddd\lambda (L_\lambda v) u.
$$
Two last terms disappear by Proposition \ref{prop:dlambda} and we finally get
$$
[L_\lambda u, L_\lambda v ] = L_\lambda \bigl([Lu,v] + [u, Lv] - (L+\lambda\Id)[u,v]  \bigr) \in \Image
L_\lambda.
$$
In view of the Frobenius theorem, this completes the proof. \end{proof}

\begin{Corollary} 
\label{cor:image_k}
In the assumptions of Proposition \ref{prop:image}, suppose in addition that the distributions  $\Image \bigl( L -\lambda(x)\cdot \Id\bigr)^k$, $k\in\mathbb N$,  are all of constant rank.  These distributions are all smooth and integrable. If, in addition, $\lambda=\lambda(x)$ is the only eigenvalue of $L$ at $x$, then there exists a coordinate
system in which $L$ takes an upper-triangular form (with $\lambda$ on the diagonal).

\end{Corollary}

\begin{proof} For $k=2$, consider the restriction of
$L$ to the foliation denerated by $\Image L_\lambda$, where $L_\lambda = L -\lambda(x)\cdot \Id$, and apply Propositions \ref{prop:quotient1} and \ref{prop:image} for the Nijenhuis operator $L|_{\Image L_ \lambda}$. Then proceed by induction.    To prove the second statement,  take the adapted coordinate system for the flag of integrable distributions 
$$
\{0\}\subset {\Image L_ \lambda^m} \subset  {\Image L_ \lambda^{m-1}} \subset \dots \subset
{\Image L_ \lambda} \subset TM,
$$  
where $m$ is such that ${\Image L_ \lambda^m} \ne \{0\}$ and $\Image L_ \lambda^{m+1}=\{0\}$. \end{proof}

In general, it is not true that $\Ker L$ defines an integrable
distribution. The following example is due to Kobayashi \cite{kob}. 
 
\begin{Ex} {\rm Consider $\R^4$ with coordinates $x^i$, $i = 1 \dots 4 $. Take the following vector fields:
$$
v_1 = \pd{}{x^1}, \quad 
 v_2 = \pd{}{x^2}, \quad 
 v_3 = \pd{}{x^3} \ \ \mbox{and} \  \
 v_4 = \pd{}{x^4} + (1 + x^3) \pd{}{x^1}. \\
$$
Define operator $L$ acting  as follows $L v_1 = v_2$, $L v_i = 0$, $i = 2 \dots 4$. For this operator $L^2 = 0$ and ${{\mathcal N}_L} = 0$. At the same time the kernel is spanned by vector fields $v_2, v_3, v_4$ and $[v_3, v_4] = v_1$ so that the distribution $\Ker L$ is not integrable.  However, we still have $[\Ker L, \Ker L ]\subset \Ker L^2$. The next proposition shows that this is always the case.}
\end{Ex}

\begin{Proposition}
Let $L$ be a Nijenhuis operator. 
If $u,v \in \Ker L$ then $[u,v]\in \Ker L^2$. Similarly, if
$u,v\in \Ker \bigl(L-\lambda(x)\cdot\Id\bigr)$, then $[u,v]\in \Ker \bigl(L-\lambda(x)\cdot\Id\bigr)^2$, where $\lambda(x)$ is a smooth eigenvalue of $L$.
\end{Proposition}

{\it Proof.}
If $Lu=0$ and $Lv=0$, then 
$$
0=\mathcal N_L(u,v)=L^2 [u,v] - L[Lu,v] - L[u,Lv] + [Lu,Lv] = L^2 [u,v],
$$
as needed. Let $Lu=\lambda u$ and $Lv=\lambda v$.
Then we have
$$
0=L^2[u,v] + [\lambda u,\lambda v] - L[\lambda u,v]- L[u,\lambda v]=
$$
$$
L^2[u,v] + \lambda^2 [u,v] + \lambda \ddd\lambda (u) v -
\lambda \ddd\lambda (v) u - \lambda L [u,v] +  \ddd\lambda (v) Lu - \lambda
L[u,v] - \ddd\lambda (u) Lv=  
$$
$$
L^2[u,v] - 2\lambda L[u,v]+ \lambda^2 [u,v] =
(L - \lambda\cdot\Id)^2 [u,v].
$$
In other words, $[u,v]\in \Ker L_\lambda^2$, as needed. \hfill $\square$


\section{Analytic matrix functions, splitting and complex Nijenhuis operators}\label{sect:3}

\subsection{Analytic functions of Nijenhuis operators}
\label{analytic}

Here we generalize Proposition \ref{prop:inv_form1} to a class of (matrix) functions  much more general than polynomials. 
Suppose a compact set   $K \subseteq \mathbb{C} $,  a function $f: {K} \to \mathbb{C} $ and an operator $L$  (on $M$)  satisfy the following assumptions:  

\begin{enumerate}[(i)] 

\item $\mathbb{C}\setminus K$ is connected (we do not require that $K$  is connected);\label{i}

 \item  $f:K\to \mathbb{C}$ is  a continuous function, and  the restriction $f |_{\mathrm{int}(K)}$ is holomorphic; \label{ii}

\item $K$ is symmetric with respect to the $x$-axes: for every  $z\in K$ its conjugate $\bar z $ also belongs to $K$. \label{iii}
 
 \item  $f(z)= \bar{f}(\bar z)$ for every $z\in K$;  \label{iiii}
  
  \item  for  every $x\in M$  we have $\mathrm{Spectrum}_x L \subset \mathrm{int}(K)$. \label{iiiii} 
  
  \end{enumerate}

Typical examples of such functions are  polynomials with real coefficients,    the functions  $e^z$, $\cos(z)$, $\sin(z)$, etc.
Under the above assumptions (\ref{i}--\ref{iiiii}), one can naturally define  an 
operator  $f(L)$.  Indeed, by  Mergelyan's theorem \cite{mergelyan}, the function $f$  can be uniformly approximated by polynomials $p_k$.  Moreover, we can choose these polynomials in such a way that  $p_k(z)= \bar p_k(\bar z)$ for every $z\in K$, i.e., their coefficients are real. 
We define $f(L)= \lim_{k\to \infty }p_k(L)$. 

It is an easy exercise  (see 
  for example   \cite[\S 1.2.2 -- 1.2.4]{higham}) to show that  the limit  exists,  is independent on the choice of the sequence $p_k$,   smoothly depends on $x\in M^m$ (actually, the function $f(L)$ is analytic in the entries of $L$), and behaves as a $(1,1)$-tensor under coordinate transformations.  Moreover,     
 for  every $x\in M $ we have: $\mathrm{Spectrum}_x f(L) =f(\mathrm{Spectrum}_x L )$.    In particular, if for  every  $x\in M$ we have $0\not\in f(\mathrm{Spectrum}_x L)$, then the tensor $f(L)$ is non-degenerate.  Finally, since uniform convergence of holomorphic functions implies convergence of derivatives, we also have $\mathcal L_\xi \bigl( f(L) \bigr) = \lim_{k\to\infty}\mathcal L_\xi \bigl( p_k(L)\bigr)$ where $\mathcal L_\xi$ denotes the Lie derivative  (in particular, a similar relation holds for partial derivatives of the entries of $f(L)$).  The latter property immediately implies

\begin{Proposition}
\label{prop:lem3} 
Let $L$  be a Nienhuis operator. Then $f(L)$ is a Nijenhuis operator for any  function $f:{K}\to \mathbb{C}$ satisfying  assumptions {\rm (\ref{i})--(\ref{iiiii})}.
\end{Proposition}

One example of such functions, which is an obvious generalization of the matrix sign function \cite{higham}, is particularly important for us. 

\begin{Ex}  
\label{ex:Bols2}
{\rm Consider the characteristic polynomial $\chi(t)$ of $L$
$$
\chi:M\times \mathbb{R}\to \mathbb{R}, \  \  (x,t)\mapsto  \chi(t):= \det\bigl(t\cdot \Id-L(x)  \bigr),
$$
as a monic polynomial in $t$ of degree $n$ whose coefficients are smooth functions on $M$.

We will say that a factorisation $\chi(t)=\chi_1(t)\cdot \chi_2(t)$  is {\it admissible}, if  $\chi_1, \chi_2:M\times \mathbb{R}\to \mathbb{R} $  are monic polynomials of degree $\ge 1$ with smooth coefficients, which are coprime  at every point $x\in M$ (i.e., have no common roots).

Consider the following two distributions $\mathcal D_1, \mathcal D_2$ on $M$:
\begin{equation}
\label{D_i}  \mathcal D_i=\Ker \chi_i(L)\ , \ \ i=1,2 , 
\end{equation}
where $\Ker$ denotes the kernel. 
It is easy to see that  $\mathcal D_1$ and $\mathcal D_2$ are  transversal $L$-invariant
distributions  of complementary dimensions   on $M$, so that   $T M= \mathcal D_1\oplus \mathcal D_2$.    
This decomposition  defines two natural
projectors $P_1$ and $P_2$ onto the distributions $\mathcal D_1$ and $\mathcal D_2$
respectively.  A simple but important observation is that (locally)  these
projectors can be viewed as analytic functions $P_1=f_1(L)$ and
$P_2=f_2(L)$ satisfying the above assumptions (\ref{i})--(\ref{iiiii}).

Indeed, we take a point $x\in M$ and consider the zeros $\lambda_1,\dots ,\lambda_r$   and   
$\lambda_{r+1},\dots ,\lambda_{n}$ of the polynomials   $\chi_1$ and respectively   $\chi_2$ at the point. 
We take small positive $\varepsilon$ and consider 
$$
K_1:= \bigcup_{i=1}^{r} B(\lambda_i, \varepsilon)\ , \  \  K_2:= \bigcup_{i={r+1}}^{n} B(\lambda_i, \varepsilon)\ , \  \ K:= K_1\cup K_2,
$$  
 where  $B(z, \varepsilon)\subset \mathbb{C}$ denotes a closed ball of radius $\varepsilon$ centered at $z\in \mathbb{C}$. We can choose $\varepsilon$ small enough so that $K_1 \cap K_2 = \emptyset$ and $K$ satisfy the assumptions (\ref{i})  and  (\ref{iii}). If we work in a small neighbourhood of $x$, the assumption (\ref{iiiii}) is evidently fulfilled. 

Now take the function $f_1:K\to \mathbb{C}, \  \ f_1(z):= \left\{\begin{array}{cc} 1 & \textrm{ for $z\in K_1$ } \\ 0 & \textrm{ for $z\in K_2$. } \end{array} \right. $  
 The function  evidently  satisfies (\ref{ii}, \ref{iiii}). Then $f_1(L)$ is well-defined. It is easy to see that $f_1(L)(\xi)= \xi$ for $\xi\in \Ker\chi_1(L)$, and $f_1(L)(\xi)=0$ for $\xi\in \Ker \chi_2(L)$, i.e., $f_1(L)$ is the projector $P_1$.      Similarly, one can construct a function $f_2$ such that $f_2(L)= P_2$. 

Thus,  by Proposition \ref{prop:inv_form1},  ${\mathcal N}_{P_i}=0$, $i=1,2$.  Similarly,  the projector to every generalised eigensubspace of a Nijenhuis operator $L$ is  a Nijenhuis operator  itself (whenever such a projector is well defined). 
}\end{Ex}    

Another important example of an analytic function of a Nijenhuis operator is discussed in Section  \ref{noreal}.

\subsection{Splitting theorem}
\label{splitting}

Consider a point $p\in M$ and let $\lambda_1, \dots, \lambda_n$ be the eigenvalues (possibly complex) of $L(p)$ counted with multiplicities.   Assume that they are divided into two groups $\{\lambda_1,\dots,\lambda_r\}$ and $\{\lambda_{r+1},\dots,\lambda_n\}$ in such a way that equal eigenvalues as well as complex conjugate pairs belong to the same group, i.e., for every $i\in \{1,\dots,r\}$, $j\in \{r+1,\dots,n\}$ we have $\lambda_i\not=\lambda_j$ and $\lambda_i\not=\bar\lambda_j$.  Clearly, at the point $p$ we have  an admissible factorisation of the characteristic polynomial
$\chi_{L(p)} (t) = \chi_1(t)\, \chi_2(t)$ (see Example \ref{ex:Bols2}), where 
$$
\chi_1=(t-\lambda_1)\ldots
(t-\lambda_{r}),\quad   \chi_2=(t-\lambda_{r+1})\ldots
(t-\lambda_{n}).
$$
By continuity, this partition of eigenvalues can be extended to a certain neighbourhood $U(p)$ which leads to an admissible factorisation of the characteristic polynomial $\chi_L(t)$ on $U(p)$ given by the same formula.  The fact that the coefficients of $\chi_1(t)$ and  $\chi_1(t)$ are smooth on $U(p)$ easily follows from the Implicit Function Theorem.  All local admissible factorisations are of this kind. Notice that we do not assume that $p\in M$ is algebraically generic, i.e.,  $p$ is allowed to be singular.  

Following Example \ref{ex:Bols2}, we consider the distributions $\mathcal D_i = \Ker \chi_i(L)$ ($i=1,2$) that provide a natural decomposition of the tangent bundle $TM = \mathcal D_1 \oplus \mathcal D_2$.

\begin{Theorem} \label{thm1}
Let $\chi_L(t)=\chi_1(t)\, \chi_2(t)$ be an admissible factorisation of the characteristic polynomial of a Nijenhuis operator $L$ in a neighbourhood of a point $p\in M$.  Then the distributions $\mathcal D_i = \Ker \chi_i(L)$  are both  integrable.   Moreover, in any adapted coordinate system $(x_1,\ldots, x_r,y_{r+1}, \ldots, y_{n})$ {\rm(}i.e., such that  $\mathcal D_1$ is generated by the vectors $\partial_{x_i}$, and  $\mathcal D_2$ is generated by the vectors $\partial_{y_j}${\rm)} the operator $L$ takes the block-diagonal form
 \begin{equation}
 \label{matl} 
L(x,y)=\begin{pmatrix} L_1(x) & 0 \\  0 & L_2(y)\end{pmatrix},
\end{equation}
where the entries of the $r\times r$ matrix $L_1$ depend on the $x$-variables only, and the entries of the $(n-r)\times (n-r)$ matrix $L_2$ depend on the $y$-variables only. In other words,
$L$ splits into a direct sum of two Nijenhuis operators: $L(x,y) = L_1(x) \oplus L_2(y)$. 
 \end{Theorem}

\begin{proof} 
Consider the projectors $P_1$ and $P_2$ onto the distributions $\mathcal D_1$ and $\mathcal D_2$ 
respectively defined by the decomposition $TM=\mathcal D_1\oplus\mathcal D_2$.   Since $P_i$ is a Nijenhuis operator  (see Example \ref{ex:Bols2}) and 
$\mathcal D_i = \mathrm{Image}\,P_i$,  then the integrability of $\mathcal D_i$ follows from  Proposition \ref{prop:image}.

The integrability of $\mathcal D_1$ and $\mathcal D_2$ is equivalent to the
existence of an adapted local coordinate system $(x,y)$ such that
$\partial_{x_1},\dots, \partial_{x_{r}}$ generate $\mathcal D_1$ and
$\partial_{y_{r+1}},\dots, \partial_{y_{n}}$ generates $\mathcal D_2$. Since these distributions are both $L$-invariant, the operator $L$ in this coordinate system has a
block diagonal form:
$$
L(x,y)=\begin{pmatrix} L_1(x,y) & 0 \\ 0 & L_2(x,y) \end{pmatrix}.
$$

Without loss of generality we may assume that $\det L\ne 0$, otherwise we can (locally)  replace $L$ by $L+ c\cdot {\Id}$  where $c$ is an appropriate constant.

Notice that the operator
$$
L P_1=\begin{pmatrix} L_1(x,y) & 0 \\ 0 & 0
\end{pmatrix},
$$
being a function of $L$,  has zero Nijenhuis tensor.
Thus, for $u=\partial_{y_\alpha}\in D_2=\Ker P_1$ we have
$$
\begin{aligned}
\begin{pmatrix} 0 & 0 \\ 0 & 0 \end{pmatrix}=\mathcal L_{LP_1u}(LP_1) &- LP_1\mathcal L_u(LP_1)= LP_1\mathcal
L_u(LP_1)= \\ 
& \begin{pmatrix} L_1 & 0 \\ 0 & 0
\end{pmatrix}
\begin{pmatrix} \partial_{y_\alpha}L_1 & 0 \\ 0 & 0
\end{pmatrix}=
\begin{pmatrix} L_1\partial_{y_\alpha}L_1 & 0 \\ 0 & 0
\end{pmatrix}
\end{aligned}
$$

Since $L_1$ is non-degenerate we conclude that
$\partial_{y_\alpha}L_1=0$, i.e., $L_1=L_1(x)$. Similarly,
$L_2=L_2(y)$, as needed. \end{proof}

By induction, we come to the following conclusion.

\begin{Theorem}
\label{th:bols3.2}
Assume that the spectrum of a Nijenhuis operator  $L$ at a
point $p\in M$ consists of $k$ real (distinct) eigenvalues $\lambda_1,\dots, \lambda_k$  with multiplicities $\sfm_1,\dots, \sfm_k$  respectively and $s$ pairs of complex (non-real) conjugate eigenvalues $\mu_1,\bar \mu_1, \dots ,\mu_s,\bar \mu_s$ of multiplicities $\sfl_1,\dots, \sfl_s$.  Then in a neighbourhood of $p\in M$ there exists a local coordinate system 
$$
x_1=(x_{1}^1\dots  x_{1}^{\sfm_1}),  \dots  , x_k=(x_{k}^1 \dots  x_{k}^{\sfm_k}),\quad 
u_1=(u_{1}^1 \dots  u_{1}^{2\sfl_1}),  \dots  , u_s=(u_{s}^1 \dots  u_{s}^{2\sfl_s}),
$$
in which $L$ takes the following block-diagonal form
$$
L = 
\begin{pmatrix} 
L_1(x_1)&  & & & & \\
 &\!\!\! \ddots & & &  &\\
 & & \!\! L_k(x_k)& & & \\
 & & & \!\!\!\!\!\! L^{\mathbb C}_1 (u_1)&  & \\
 & & & & \!\!\! \ddots  & \\
 & & & & & \!\! L^{\mathbb C}_s (u_s)
\end{pmatrix}
$$
where each block depends on its own group of variables and is a Nijenhuis operator w.r.t. these variables.   In other words, every Nienhuis operator  $L$  locally  splits into a direct sum of Nijenhuis operators $L=L_1\oplus L_2 \oplus \dots $  each of which at the point $p\in M$  has either a single real eigenvalue  or a single pair of complex eigenvalues.   
\end{Theorem}
Notice that within each block, the algebraic structure may vary from point to point as $p$ is not supposed to be algebraically generic.

\subsection{Nijenhuis operators with no real eigenvalues}
\label{noreal}

In this section we study real Nijenhuis operators all of whose eigenvalues are complex, that is, the spectrum of $L$ at every point $p\in M$  
consists of pairs of complex conjugate eigenvalues $\mu_i, \bar \mu_i$,  $\mu_i\notin \R$ or, equivalently, the spectrum of $L$ belongs to $\mathbb C \setminus\R$.  Our goal is show that in this case $M$ can be endowed with a canonical complex structure $J$ with respect to which $L$ is a complex holomorphic Nijenhuis operator, i.e.,  all the entries $l^i_j(z)$ of $L$, as a complex matrix, are holomorphic functions of complex coordinates $z^1, \dots, z^n$ and the  {\it complex Nijenhuis tensor} of $L$ defined in Theorem \ref{thm:comlexcase} (3) vanishes.

Let $L$ be a operator (not necessarily Nijenhuis) on a real manifold $M$ with no real eigenvalues.  This implies that the dimension of $M$ is even since the characteristic polynomial $\chi_L(t)$ has no real roots and therefore must be of even degree.  We first describe the canonical complex structure $J$ associated with $L$ and constructed as an analytic function $f(L)$ (satisfying conditions (\ref{i})--(\ref{iiiii}) from Section \ref{analytic}).  

Consider  the following function on $\mathbb C \setminus \R$:
  $$
  f(z) = \left\{  \begin{array}{rl} 
  \ii \ , & \mbox{if } \mathrm{Im}\, z >0, \\
  - \ii \ , & \mbox{if } \mathrm{Im}\, z <0. 
  \end{array}
  \right.
  $$
This function is locally constant and holomorphic on $\mathbb C \setminus \R$. It is easy to see that all properties (\ref{i})--(\ref{iiiii}) are fulfilled for any symmetric compact subset $K \subset \mathbb C\setminus \R$. 
 
 \begin{Proposition}   
 \label{prop:3.2} 
 Let $L$ be an operator on a real manifold $M$ with no real eigenvalues. Then 
  the operator 
  \begin{equation}
\label{eq:f(L)}
J = f(L)
\end{equation} 
is well defined on $M$ and satisfies the following properties:
\begin{enumerate}
  \item $J$ is  smooth;
  \item $J^2 = -\Id$, i.e. $J$ is an almost complex structure on $M$;
  \item $JL=LJ$, i.e. $L$ is a complex linear operator w.r.t. $J$;
  \item if $L$ is Nijenhuis, then  $J$  is integrable and hence is a complex structure on $M$. 
 \end{enumerate}
\end{Proposition}

A very special case of item 4 was discussed in \cite{MatveevNij}.
  
\begin{proof} The fact that $J=f(L)$ is well defined and is smooth  follows from Section \ref{analytic}.
Next, items 2 and 3 are purely algebraic.  Indeed,  the scalar identity $f^2(z)\equiv -1$ implies the matrix identity $J^2 = f(L) f(L) = - \Id$, and $L$ commutes with $f(L)$  for any matrix function $f$. Finally, the integrability of $J$, i.e.,  the fact that ${\mathcal N}_J\equiv 0$, is a particular case of Proposition \ref{prop:lem3}. \end{proof}

  \begin{Theorem}
  \label{thm:comlexcase}
  Let  $L$ be a Nijenhuis operator  on $M$ with no real eigenvalues,  i.e., its spectrum at every point $x\in M$ belongs to $\mathbb C \setminus\R$. Then 
  \begin{enumerate}
  \item $M$ is a complex manifold w.r.t. $J$ canonically associated with $L$ by \eqref{eq:f(L)}. 
  
 \item  $L$ is a complex holomorphic tensor field on $M$ w.r.t. $J$.  In other words, the subspace spanned by $\partial_{z^1}, \dots, \partial_{z^n}$ is  $L$-invariant and the restriction of $L$  to this subspace takes the form
  $$
  L^{\mathbb C} =  \sum_{i,j=1}^nl^i_j(z)\, \ddd z^j \otimes \partial_{z^i}
  $$  
  with all the functions $l^i_j(z)$ being holomorphic.
    
\item  The complex Nijenhuis  tensor of $L$ vanishes, i.e. 
  \begin{equation}
  \label{eq:comlNijrel}
\left( \mathcal N_L^{\mathbb C} \right)^i_{jk} = l^m_j \pd{l^i_k}{z^m} - l^m_k \pd{l^i_j}{z^m} - l^i_m \pd{l^m_k}{z^j} + l^i_m \pd{l^m_j}{z^k} =0.
  \end{equation}
\end{enumerate}
  \end{Theorem}

\begin{proof}
The first statement follows from Proposition \ref{prop:3.2}. To prove statements 2 and 3, we will use complex coordinates $z^1, \dots, z^n$ associated with $J$.  

At each point $p\in M$  we consider $L$ as a linear operator $L: T_p M \to T_p M$. In the basis  $\partial_{z^i}, \partial_{\bar z^i}$ ($i=1,\dots, n$) every operator can be written in the form\footnote{Speaking more formally, in this formula we think of $L$ as an operator acting on the complexified tangent space, i.e,  $2n$-dimensional vector space over $\mathbb C$ with basis $\partial_{z^1}, \dots, \partial_{z^n}, \, \partial_{\bar z^1},\dots , \partial_{\bar z^n}$.}
$$
L  = l^i_j (p) \, \partial_{z^i}\otimes \ddd z^j +  m^i_j (p)\, \partial_{\bar z^i}\otimes \ddd z^j +  \bar m^i_j (p)\, \partial_{z^i}\otimes \ddd \bar z^j +  \bar l^i_j (p) \, \partial_{\bar z^i}\otimes \ddd \bar z^j
$$
with some complex-valued functions $l^i_j(p), m^i_j(p)$ and summation over $i$ and $j$ assumed. By Proposition \ref{prop:3.2},  $L$ commutes with $J$, i.e., is a $J$-linear.  In terms of the above decomposition this means that $m^i_j(p)\equiv 0$. If we now think of $L$ as a complex linear map acting on   $\mathrm{span}(\partial_{z^1}, \dots, \partial_{z^n})$,  we may associate it with the $n\times n$ complex matrix $L^{\mathbb C}=\bigl( l^i_j (p)   \bigr)$.
Our goal is to show that its entries $l^i_j(p)=l^i_j(z^1,\dots, x^n)$ are holomorphic in coordinates $z^1, \dots, z^n$.

\begin{Lemma}
Let $L$ be a field of complex endomorphisms on a complex manifold $(M, J)$. Assume that the Nijenhuis tensor of $L$  vanishes in the real sense, $L$ has no real eigenvalues and, in addition, the imaginary part of each eigenvalues is positive.  Then $L$ is holomorphic and its complex Nijenhuis tensor vanishes.
\end{Lemma}

\begin{proof}
As explained above, if we think of $L$ as a real operator acting on the $2n$-dimensional tangent space, then in the basis
$\partial_{z^1},\dots,\partial_{z^n}, \partial_{\bar z^1},\dots,\partial_{\bar z^n}$  we have
$L  = l^i_j (p) \, \partial_{z^i}\otimes \ddd z^j  +  \bar l^i_j (p) \, \partial_{\bar z^i}\otimes \ddd \bar z^j$.  Since the real Nijenhuis tensor $\mathcal N_L$ vanishes, we have:
$$
L^2[\partial_{z^i}, \partial_{\bar z^j}] - L [ L\partial_{z^i}, \partial_{\bar z^j} ] -
 L [ \partial_{z^i}, L\partial_{\bar z^j} ] +   [ L\partial_{z^i}, L\partial_{\bar z^j} ] =0.
$$

Computing formally the left hand side and collecting the terms containing  $\partial_{z^\beta}$, we obtain the following relation:
$$
\left( l_\alpha^\beta \frac{\partial l^\alpha_i}{\partial \bar z^j}  - \bar l^\alpha_j \frac{\partial l^\beta_i}{\partial \bar z^\alpha}\right) \partial_{z^\beta} =0.
$$

Fixing the index $i$, we can treat it as a matrix identity of the form
$$
L^{\mathbb C} X - X\overline { L^{\mathbb C} }=0,
$$
where $X^\alpha_j = \frac{\partial l^\alpha_i}{\partial \bar z^j}$ and $(L^{\mathbb C})^\beta_\alpha=  l^\beta_\alpha$.
Equations of this kind   (i.e., $AX-XB=0$ with given $A$ and $B$ and unknown $X$) are well known in Linear Algebra.  
One of the main properties of such equations is as follows: if $A$ and $B$ have no common eigenvalues then $X=0$.

In our case, the eigenvalues of  $L^{\mathbb C}$ and $\overline { L^{\mathbb C} }$ are complex conjugate. Moreover, according to our assumptions,  the eigenvalues of  $L^{\mathbb C}$  lie in the upper half plane, whereas those of $\overline { L^{\mathbb C} }$ in the lower half plane. In particular, $L^{\mathbb C}$ and $\overline { L^{\mathbb C} }$ have no common eigenvalues. Hence $X=0$, that is,  $\frac{\partial l^\alpha_i}{\partial \bar z^j}=0$, meaning that $L$ is holomorphic.

The condition ${{\mathcal N}_L}^{\mathbb C}=0$ given by \eqref{eq:comlNijrel} is equivalent to the relation
$$
L^2[\partial_{z^i}, \partial_{z^j}] - L [ L\partial_{z^i}, \partial_{z^j} ] -
 L [ \partial_{z^i}, L\partial_{z^j} ] +   [ L\partial_{z^i}, L\partial_{z^j} ] =0, \quad i,j=1\dots,n,
$$
which is simply a part of the real Nijenhuis relations.  \end{proof}

To complete the proof of Theorem \ref{thm:comlexcase} it remains to notice that positivity of the imaginary part of each eigenvalue of $L^\mathbb C$ is a corollary of formula \eqref{eq:f(L)} that defines $J$.  \end{proof}

\begin{Remark} 
\label{rem:fromrealtocompl}
{\rm
Theorem  \ref{thm:comlexcase} basically explains how we should treat Nijenhuis operators with complex eigenvalues (in dimension $2n$). First,  we need to introduce complex coordinates and next we may deal with such operators just in the same way as we would do in the case of real eigenvalues (in dimension $n$) having in mind that all the objects are now not just smooth, but holomorphic. For instance, if we can reduce a Nijenhuis operator to a certain canonical form in the case of real eigenvalues, then replacing real coordinates $x_1,\dots, x_n$ with complex coordinates $z_1,\dots, z_n$ we will get automatically a canonical form in the complex case provided all the steps of the reduction procedure admit natural complex analogs (which is almost always the case).  We will apply this principle below in Remark \ref{rem:canformgeneral} to get a canonical form for a complex Jordan block.
}\end{Remark}


\section{ Canonical forms for Nijenhuis operators}\label{sect:4}

\subsection{Semisimple case}

It is a well known fact that pointwise diagonalisable Nijenhuis operators can be diagonalised in a neighbourhood of a generic point. 
Let us emphasise that some kind of {\it genericity} assumption is essential  (see Example \ref{ex:4.1} below). First we recall two classical  diagonalisability theorems in a neighbourhood of a $\mathrm{gl}$-regular or algebraically generic point (see Definitions \ref{def:alggen} and \ref{def:algregular}).  One usually assumes that $L$ is diagonalisable over $\R$  (see \cite{haant}, \cite{nij}),  in our version complex eigenvalues are also allowed.

\begin{Theorem}\label{thm:part1:5}
Let $L$ be a Nijenhuis operator. Assume that L is semisimple and $\mathrm{gl}$-regular at a point $p\in M$, in other words $L(p)$  has $n=\dim M$ distinct eigenvalues (real or complex) 
$$
\lambda_1, \dots, \lambda_s \in \R \quad\mbox{and}\quad  \rho_1,\bar\rho_1, \dots, \rho_m,\bar\rho_m \in\mathbb C,  \quad s+2m=n.
$$  
Then there exists a local coordinate system 
$$
u_1,\dots , u_s, \ x_1, y_1, \dots , x_m, y_m \qquad (\mbox{we set $z_\alpha=x_\alpha + \iii y_\alpha$})
$$ 
in which  $L$ takes the following block-diagonal form with $1\times 1$ blocks corresponding to real eigenvalues and $2\times 2$ blocks corresponding to pairs of complex eigenvalues:
$$
L=\mathrm{diag} \Bigl( \lambda_1(u_1), \dots, \lambda_s(u_s),   R_1(z_1) , \dots , R_m(z_m) \Bigr) \quad\mbox{with }
R_\beta(z_\beta)  = \begin{pmatrix}
a_\beta(z_\beta) & \!\!\! - b_\beta(z_\beta) \\
b_\beta(z_\beta) & \! a_\beta(z_\beta)
\end{pmatrix},
$$
where $\lambda_\alpha(u_\alpha)$ is smooth   
and $\rho(z_\beta)=a_\beta(z_\beta) + \iii b_\beta(z_\beta)$ is a holomorphic function in  $z_\beta$.  
\end{Theorem}

\begin{proof}
This statement is a particular case of Theorem \ref{th:bols3.2}   (for each complex $2\times 2$-block,  we choose complex coordinates following Theorem \ref{thm:comlexcase}).   
\end{proof}

This theorem can be easily generailsed to the semisimple case with multiple eigenvalues under additional assumption that $L$ is algebraically generic, i.e., multipliciteis of eigenvalues are locally constant. 

\begin{Theorem}\label{thm:part1:6}
Let $L$ be a semisimple Nijenhuis operator.  Assume that $L$ is algebraically generic at a point $p\in M$ and its real eigenvalues $\lambda_1, \dots, \lambda_s \in \R
$ have multiplicities $k_1, \dots, k_s$ and pairs of complex eigenvalues 
$\rho_1,\bar\rho_1, \dots, \rho_m,\bar\rho_m \in\mathbb C$  have multiplicities $l_1,\dots, l_m$ respectively. 
Then there exists a local coordinate system
$$
\underbrace{u_1^1,\dots u_1^{k_1}}_{u_1}, \dots , \underbrace {u_s^1,\dots, u_s^{k_s}}_{u_s}, \ \
\underbrace {z_1^1,\dots, z_1^{l_1}}_{z_1}, \dots , \underbrace {z_m^1,\dots, z_m^{l_m}}_{z_m},  \quad\mbox{where } z^\alpha_\beta=
x^\alpha_\beta + \iii y^\alpha_\beta,
$$ 
in which  $L$ takes the following block-diagonal form: 
$$
L=\mathrm{diag} \Bigl( \Lambda_1(u_1), \dots, \Lambda_s(u_s),  \ \mathcal R_1(z_1) , \dots , \mathcal R_m(z_m) \Bigr) 
$$
where  each block corresponds to one of real eigenvalues or one of pairs of complex eigenvalues, namely, 
$$
\Lambda_\alpha(u_\alpha) =\mathrm{diag} \Bigl( \underbrace{\lambda_\alpha(u_\alpha),\dots,  \lambda_\alpha(u_\alpha)}_{k_\alpha  \ \mathrm{times}}\Bigr),  \qquad
\mathcal R_\beta(z_\beta) =\mathrm{diag} \Bigl( \underbrace{R_\beta(z_\beta),\dots,  R_\beta(z_\beta)}_{l_\beta  \ \mathrm{times}}\Bigr),
$$
with 
$$
R_\beta(z_\beta)  = \begin{pmatrix}
a_\beta(z_\beta) & \!\!\! - b_\beta(z_\beta) \\
b_\beta(z_\beta) & \! a_\beta(z_\beta)
\end{pmatrix},
$$  

where $\lambda_\alpha(u_\alpha)$ is a smooth function of $k_\alpha$ coordinates $u_\alpha = (u^1_\alpha, \dots ,u^{k_\alpha}_\alpha)$ 
and $\rho(z_\beta)=a_\beta(z_\beta) + \iii b_\beta(z_\beta)$ is a holomorphic function of $l_\beta$ complex variables   $z_\beta=(z_\beta^1,\dots, z_\beta^{l_\beta})$.  
\end{Theorem}

In the next theorem we do not assume that $L$ is $\mathrm{gl}$-regular or algebraically generic at a point $p\in M$.  Instead we suppose that the eigenvalues  of $L$, both real and complex,  are all smooth and functionally independent in a neighbourhood of $p\in M$. Notice that here we do not assume that the eigenvalues are all distinct at the point $p$, it may well happen that all of them vanish at $p$ simultaneously so that $p$ is a singular point of scalar type (Definition \ref{def:scalar}).

\begin{Theorem}
\label{thm:4.3}
Assume that in a neighbourhood of a point $p\in M^n$, the eigenvalues of a Nijenhuis operator $L$
$$
\lambda_1, \dots, \lambda_s \in \R \quad\mbox{and}\quad  \rho_1= \alpha_1 \pm  \ii \beta_1   , \dots,   \rho_m= \alpha_m \pm  \ii \beta_m   \in \mathbb C,  \quad n=s + 2m,
$$
are smooth and functionally independent functions (the latter condition means that the differentials of the $n$ functions 
$\lambda_1,\dots, \lambda_s, a_1, b_1, \dots, a_m, b_m$ are linearly independent at the point $p\in M$).  Then we can introduce a local coordinate system $u_1,\dots, u_s, x_1, y_1, \dots , x_m, y_m$  by setting $u_k=\lambda_k$, $x_j=\alpha_j$ and $y_j=\beta_j$ and in this coordinate system the operator $L$ takes the following form:
\begin{equation}
\label{eq:diagL1}
L=\mathrm{diag}\Bigl(  u_1, \dots, u_s , R_1, \dots, R_m \Bigr)\quad \mbox{with } R_j = \begin{pmatrix} x_j & \!\!\! - y_j\\ y_j & x_j \end{pmatrix}.
\end{equation}
\end{Theorem}

\begin{proof}
The right hand side of \eqref{eq:diagL1}  defines a Nijenhuis operator whose characteristic polynomial coincides with that of $L$ and the coefficients of this polynomial are functionally independent almost everywhere. Hence, the statement follows from Corollary \ref{cor:Bols1}. 
\end{proof}

\begin{Ex}
\label{ex:4.1}
{\rm
Notice that functional independence of eigenvalues in Theorem \ref{thm:4.3}   (as well as algebraic regularity and algebraic genericity assumptions in Theorems \ref{thm:part1:5} and \ref{thm:part1:6})  is essential even if we assume that $L$ is semisimple.     Here is one of the simplest examples:  
$$
L=\begin{pmatrix}    x^2 &  xy \\ xy & y^2 \end{pmatrix}.
$$    
The eigenvalues of $L$ are $\lambda_1 = 0$ and $\lambda_2=x^2 + y^2$,  and $L$ in Nijenhuis. However $L$ cannot be diagonalised in a neighbourhood of $(0,0)$.
}\end{Ex}

\subsection{Differentially non-degenerate Nijenhuis operators}

The next two theorems give canonical forms under a rather different condition that $L$ is differentially non-degenerate (Definition \ref{def:nondeg}).

\begin{Theorem}\label{thm:part1:7} Let $L$ be a Nienhuis operator on $M^n$. Assume that $L$ is differentially non-degenerate at a point $p\in M$, i.e., the differentials $\ddd\sigma_1,\dots,\ddd\sigma_n$ of the coefficients of the characterteristic polynomial of $L$ are linearly independent at $p$. Then there exists a local coordinate system $x_1,\dots, x_n$ in which $L$ takes the following canonical form:
\begin{equation}\label{eq:14}
L = \begin{pmatrix}
x_1 & 1 &  & & \\
x_2 & 0 & 1 & & \\
\vdots & \vdots & \ddots & \ddots & \\
x_{n-1} &  0 & \dots & 0 & 1\\
x_n & 0  & \dots & 0 & 0 
\end{pmatrix}
\end{equation}
\end{Theorem}

\begin{proof}
Since $\sigma_1,\dots,\sigma_n$ are functionally independent, we can set $x_i = -\sigma_i$.  Then the right hand side of \eqref{eq:14} defines a Nijenhuis operator whose characteristic polynomial coincides with $\chi_L(t)$ at each point.    After this it remains to apply Corollary \ref{cor:Bols1}.
Alternatively, one can use \eqref{eq:Lexplicit} to ``compute'' $L$ in coordinates $x_1,\dots, x_n$.
\end{proof}

Notice that any operator given by \eqref{eq:14} is automatically $\mathrm{gl}$-regular in the sense of Definition \ref{def:algregular}.  However we can generalise this statement to include more general examples.

\begin{Theorem}\label{thm:part1:8}
Assume that in a neighbourhood of a point $p\in M$, the characteristic polynomial of a Nijenhuis operator $L$ factorises into several monic polynomials with smooth coefficients
$$
\chi_L(t) = \prod_{\alpha=1}^s \chi_\alpha (t), \qquad  \deg \chi_\alpha(t)= k_\alpha, \quad \sum_{\alpha=1}^s k_\alpha = \dim M.
$$
Suppose that all the coefficients of all polynomials $\chi_\alpha$'s are functionally independent \footnote{Notice that the total number of these coefficients is exactly $\dim M$.} at the point $p$ (i.e., their differentials are linearly independent at this point). Then there exists a local coordinate system 
$$
\underbrace{x_1^1,\dots x_1^{k_1}}_{x_1}, \  \underbrace{x_2^1,\dots x_2^{k_2}}_{x_2}, \ \dots , \ \underbrace{x_s^1,\dots x_s^{k_s}}_{x_s} 
$$
in which $L$ splits into direct sum of operators $L_1(x_1), L_2(x_2),\dots, L_s(x_s)$, i.e.,
\begin{equation}
\label{eq:15}
L=\mathrm{diag} \Bigl( L_1(x_1), L_2(x_2),\dots, L_s(x_s)\Bigr),
\end{equation}
where
\begin{equation}
\label{eq:16}
L_\alpha(x_\alpha) = 
 \begin{pmatrix}
x_\alpha^1 & 1 &  & & \\
x_\alpha^2 & 0 & 1 & & \\
\vdots & \vdots & \ddots & \ddots & \\
x_\alpha^{k_\alpha -1} &  0 & \dots & 0 & 1\\
x_\alpha^{k_\alpha} & 0  & \dots & 0 & 0 
\end{pmatrix}.
\end{equation}
\end{Theorem}

\begin{proof}
The proof is similar to that of Theorems \ref{thm:4.3} and \ref{thm:part1:7} above and follows immediately from Corollary \ref{cor:Bols1}.
\end{proof}

Notice that in this theorem  $L$, in general,  is neither semisimple, nor $\mathrm{gl}$-regular, nor algebraically generic. For example, if  all variables $x_\alpha^\beta$'s vanish simultaneously  at the point $p$,  then $L(p)$ is nilpotent with $s$ Jordan blocks of sizes $k_1,\dots,k_s$.  However,  at almost every  point from a neighbourhood of $p$ the operator $R$ is semisimple with distinct eigenvalues. 

In this view,  formulas \eqref{eq:15} and \eqref{eq:16} can be understood as a Nijenhuis deformation of a linear operator $L(p)$ of an arbitrary algebraic type.  And other way around, they show one of possible scenarios for degeneration of a Nijenhuis operator that is $\mathrm{gl}$-regular at generic points. In particular, it is easy to see that (in an open domain of $\R^n$) one can construct an example of a Nijenhuis operator that takes all possible algebraic types for a given dimension. 

Another important corollary of Theorem \ref{thm:4.3} is  

\begin{Corollary}\label{cor:stableblock}
Differentially non-degenerate singular points are $C^2$-stable.  
\end{Corollary}

\begin{proof}
Indeed,  the property of being differentially non-degenerate is preserved under perturbations of order 2.  It remains to notice that the canonical form \eqref{eq:14} is unique (i.e., contains no parameters) and therefore does not change under such perturbations.
\end{proof}

\begin{Remark}{\rm
We do not know whether the Nijenhuis operators given by  \eqref{eq:15}--\eqref{eq:16} are $C^2$-stable (at the origin).  The problem is that factorisability of the characteristic polynomial may not be preserved under perturbations. For general (i.e. not necessarily Nijenhuis) perturbations,  this property definitely disappears, but  it is not clear what happens under perturbation in the class of Nijenhuis operators.
}\end{Remark}

\subsection{The case of a Jordan block}

First we treat the case of a Jordan block with a constant eigenvalue $\lambda\in\R$. Without loss of generality we assume  that $\lambda=0$, i.e., $L$ is nilpotent.

\begin{Theorem}\label{thm:part1:9}
Let $L$ be a Nijenhuis operator which, at each point, is similar to a nilpotent Jordan block.  Then there exists a local coordinate system $x_1,\dots, x_n$ in which $L(x)$ is constant:
\begin{equation}
\label{eq:nilpconst}
L (x) =  \begin{pmatrix} 0 &  & & \\
1 & 0 &  & \\
 &\ddots & \ddots & \\
 & &1 & 0 
 \end{pmatrix}.
\end{equation} 
\end{Theorem}

This result was independently obtained by several authors (see \cite{Boubel}, \cite{Grifone}, \cite{Thompson}),  we suggest one more proof which will next be used for the non-constant eigenvalue case.

\begin{proof} We shall prove this statement by induction using Definition \ref{def:Nij4} of the Nijenhuis tensor $\mathcal N_L$.  For this reason,  instead of $L$ we shall work with its dual $L^*$.

As we know from Proposition \ref{prop:image}, the image of $L$ defines 
an integrable distribution of codimension 1 and, consequently, there exists
a regular function $x_1$ such that $\Image L =\Ker \dd x_1$,  or equivalently,  $L^* \dd x_1 =0$.

The next coordinate function $x_2$ must satisfy the 
relation $L^* \dd x_2 =\dd x_1$.  Let us show that such a function exists.
First of all, notice that we can always find a 1-form $\alpha$ which
satisfies $L^* \alpha=\dd x_1$ (since $\dd x_1$ belongs to the image of $L^*$).
From Definition \ref{def:Nij4} and the fact that $\mathcal N_L=0$ we have:
$$
\dd({L^*}^2\alpha) (u,v) + \dd\alpha (Lu,Lv) -
\dd(L^*\alpha) (Lu,v) - \dd(L^*\alpha) (u, Lv) = 0.
$$
Using $\dd (L^*\alpha)=\dd \dd x_1=0$ and ${L^*}^2\alpha=L^*\dd x_1=0$  we immediately
obtain:
$$
\dd\alpha (Lu,Lv)=0
$$

Since the vectors of the form $Lu$ generate the kernel of $\dd x_1$, this relation means that $\dd\alpha = \beta\wedge \dd x_1$ for some 1-form $\beta$.
If we think of $\beta$ as a differential 1-form on the leaves of the
foliation $\{ x_1=\mathrm{const}\}$ which depends on $x_1$ as parameter, then $\beta$
is closed and, consequently, can locally be presented in the form $\beta=\dd h$ where $h$ is also understood as a function on the leaves of this foliations depending on $x_1$ as a parameter. 

Now we set $\tilde\alpha=\alpha - h \dd x_1$. This new form
still satisfy $L^* \tilde \alpha=\dd x_1$ and, in addition, is closed.
Hence locally we can fund a function $x_2$ such that $\dd x_2=\tilde\alpha$.

We continue the construction by induction. Suppose we have constructed 
$x_1, \dots , x_k$ such that $L^* \dd x_{i+1} = \dd x_{i}$, $i=1,\dots,k-1$.
We are going to construct $x_{k+1}$ such that $L^* \dd x_{k+1} = \dd x_k$.
Just as before, first take $\alpha$ such that $L^* \alpha = \dd x_{k}$  (notice that such a form $\alpha$ exists for purely algebraic reasons if $k<n$,  moreover this form will be linearly independent with $\dd x_1, \dots, \dd x_k$).
Then we have
$$
0=\dd({L^*}^2\alpha) (u,v) + \dd\alpha (Lu,Lv) -
\dd(L^*\alpha) (Lu,v) - \dd(L^*\alpha) (u, Lv) =
$$
$$
\dd \dd x_{k-1} (u,v) + \dd\alpha (Lu,Lv) -
\dd \dd x_k (Lu,v) - \dd \dd x_k (u, Lv) = \dd\alpha (Lu,Lv)
$$

We see that the situation is absolutely the same as before and we literally
repeat the above construction to find $x_{k+1}$.  As a result we will obtain $n$ independent functions $x_1,\dots, x_n$ satisfying $L^*(\dd x_1)=0$ and 
$ L^* (\dd x_{k+1}) = \dd x_{k}$,  $k=1,\dots, n-1$, which is equivalent to the statement of Theorem \ref{thm:part1:9}.  
\end{proof}

It follows from this proof that local coordinates $x_1, \dots, x_n$ are not unique.  Indeed, at each step we need to solve the equation of the form  $L^* \dd f = \dd x_k$ and then set $f = x_{k+1}$.    This shows that as long as $x_k$ is chosen,  the coordinate $x_{k+1}$ is defined modulo arbitrary function  $h$ satisfying   $L^* \dd h  =0$, i.e. up to adding a function of the from $h=h(x_1)$.  In other words,  canonical coordinate systems are parametrized by $n$ functions of one variable. This observation can be reformulated in the following more geometric way.

\begin{Corollary}
In the assumptions of Theorem \ref{thm:part1:9},  consider an arbitrary regular smooth curve $\gamma(t)$ which is transversal to the  distribution $\Image L$.  Then there exists a unique local coordinate system $x_1,\dots, x_n$ in which $L$ takes the canonical form \eqref{eq:nilpconst} and  the parametric equation of $\gamma$ becomes $\gamma(t) = (t, 0, \dots, 0, 0)$.
\end{Corollary}

The next theorem gives a normal form for a Jordan block with a non-constant eigenvalue.

\begin{Theorem} 
\label{thm:part1:10}
Suppose that in a neighbourhood of a point $p\in M$, a Nijenhuis operator $L$ is algebraically generic and similar to the standard Jordan block with
a non-constant real eigenvalue $\lambda (x)$. Then there exists a local
coordinate system $x_1, \dots, x_n$ in which the matrix $L(x)$ takes the following form:
\begin{equation}
\label{eq:nonconstJB}
 L(x)=L_{\mathrm{can}}=
\begin{pmatrix} \lambda(x_1) &    &  &  &  &   \\  
       1      & \!\!\! \lambda(x_1)  &    &      &       &               \\
         \xi_3    &   1  & \!\!\!  \lambda(x_1)  &     &       &             \\
          \vdots   &     &  1   & \ddots &   &  \\
           \xi_{n-1}  &     &     &  \ddots       &   \lambda(x_1)      &  \\
            \xi_{n} &     &     &        &       1   & \!\!\!  \lambda(x_1) 
\end{pmatrix} \,  \quad \mbox{where} \quad
\begin{array}{l}
\xi_3 = -\lambda'   x_3,\\
 \xi_4 = -\lambda'  \, 2  x_4,\\
 \vdots \\
 \xi_{n-1} = -\lambda'\, (n-3) \,x_{n-1},\\
 \xi_{n} = -\lambda' \, (n-2) x_{n},\\
\end{array}
\end{equation}
and $\lambda' = \frac{\partial \lambda}{\partial x_1}$.  If $\ddd \lambda(p)\ne 0$, then in \eqref{eq:nonconstJB} we may set $\lambda(x_1)=x_1$ and $\lambda'=1$.  
\end{Theorem}

Formula \eqref{eq:nonconstJB} is analogous to those obtained in \cite{BoMa2011} and \cite{turiel} in the presence of an additional compatible algebraic structure, either Riemannian metric or symplectic form (cf. Section \ref{sect:6}). For this reason, the proofs in \cite{BoMa2011}, \cite{turiel} are not applicable in our situation.

\begin{proof}
Let $\lambda: U(p) \to \R$ be the eigenvalue of $L$ considered as a smooth function  and denote $L_\lambda = L - \lambda \cdot\Id$.  

According  to Corollary \ref{cor:image_k} we have the flag of integrable distributions 
$$
\{0\} \subset  \operatorname{Image} L_\lambda^{n-1}  \subset  \operatorname{Image} L_\lambda^{n-2}  \subset \dots \subset
 \operatorname{Image} L_\lambda \subset TM.  
$$
Notice that $\Ker L_\lambda^{n-k} = \operatorname{Image} L_\lambda^{k}$ for any $k=1,\dots,n-1$.  Let $\mathcal F_k$ denote  the foliation generated by $\Image L_\lambda^k$.

Consider a local coordinate system $y_1, \dots, y_n$ adapted to this flag, more precisely such that   
$\Ker L_\lambda^{n-k} = \operatorname{Image} L_\lambda^{k} =\mathrm{span}(\partial_{y_{k+1}},\dots, \partial_{y_n})$. In other words,  we can think of $y_{k+1}, \dots, y_n$ as coordintates on the leaves of $\mathcal F_k$  and of $y_1,\dots, y_k$ as coordinates on the (local) quotient space $M/\mathcal F_k$. In this coordinate system $L$ takes lower triangular form with $\lambda$ on the diagonal,  moreover $\lambda = \lambda (y_1)$ by Proposition \ref{prop:dlambda}:

First of all we will show that for each $L$-invariant foliation $\mathcal F_k$, we can correctly define the quotient operator $\tilde L_k$ on the (local) quotient manifold $M/\mathcal F_k$.  This property is equvalent to the following lemma   (cf. Section \ref{sect:quotient}).

\begin{Lemma}
\label{lem:forJB}
In adapted coordinates $y_1, \dots, y_n$,  let   $L_k$ denote the $k\times k$ submatrix of  $L$  composed by $L^i_j$ with $1\le i,j \le k$:  
$$
 L=
\begin{pmatrix} 
\boxed{ 
\begin{matrix}
& & \\
\ & L_k &\  \\
& & 
\end{matrix}}
  &  \begin{matrix} 0 & \dots & 0 \\
\vdots & \ddots & \vdots \\
0 & \ddots & 0 \end{matrix}
  \\
\begin{matrix} * & \dots & * \\
\vdots & \ddots & \ddots \\
* & \dots & * \end{matrix}
& \begin{matrix} \lambda & \!\!\!\ddots & \vdots \\
  \ddots  & \!\!\!\ddots & 0 \\
\dots     & *  &\lambda \\
   \end{matrix}
\end{pmatrix}.
$$
Then the entries $L_k$ depend only on the first $k$ coordinates $y_1,\dots, y_k$ and $L_k$ defines a Nijenhuis operator on $M/
\mathcal F_k$.  
\end{Lemma}

\begin{proof} In view of Proposition  \ref{prop:quotient1}, the second statement follows from the first one.
Therefore it is sufficient to verify the first property for the largest submatrix $L_{n-1}$ and then proceed by induction.

By Proposition \ref{prop:quotient1}, this property is equivalent to the following condition.
Let $\eta$ be an arbitrary vector field that is tangent to $\mathcal F_{n-1}$, i.e., $\eta \in \Ker L_\lambda$ and $\xi$ be an arbitrary vector filed that preserves $\mathcal F_{n-1}$,  i.e.,  $[\xi, \eta]  \in \Ker L_\lambda$.   We need to show that $L\xi$ preserves the foliation $\mathcal F_{n-1}$ too, that is, $[L\xi, \eta]\in \Ker L_\lambda$.

Using that $L$ is Nijenhuis we obtain:
$$
L(L-\lambda\cdot\Id)[\xi,\eta] = L[L\xi,\eta] + L[\xi,L\eta] - [L\xi,L\eta] - \lambda L [\xi,\eta] =
$$
and therefore, taking into account that $L\eta = \lambda\eta$:
$$
\begin{aligned}
&= L[L\xi,\eta] + L[\xi,\lambda \eta] - [L\xi,\lambda\eta] - \lambda L [\xi,\eta] 
= L[L\xi,\eta]  + \xi(\lambda) \cdot L\eta - \lambda [L\xi, \eta]  - L\xi(\lambda) \cdot \eta \\
&= L _\lambda [L\xi, \eta]  -  L _\lambda \xi (\lambda) \cdot \eta.
\end{aligned}
$$
The left hand side of this relation vanishes as $[\xi, \eta]  \in \Ker L_\lambda$,  the 
al derivative $L _\lambda \xi (\lambda)$ vanishes by Proposition \ref{prop:dlambda} since  $\zeta (\lambda)=0$ for any vector field $\zeta \in \Image L_\lambda$.  Hence, $[L\xi, \eta] \in \Ker L_\lambda$ as needed.  Thus, we conclude that the entries of $L_{n-1}$ do not depend on $y_n$  as required. In particular, $L_{n-1}$ defines a Nijenhuis operator on $M/\mathcal F_{n-1}$ in coordinates   $y_1,\dots, y_{n-1}$ (see Proposition \ref{prop:quotient1}). 
\end{proof}

 Notice that Lemma \ref{lem:forJB} holds true for the canonical form \eqref{eq:nonconstJB} and moreover the submatrix $(L_{\mathrm{can}})_k$ of  $L_{\mathrm{can}}$  represents a canonical form for  a Jordan block in dimension $k$.   This suggests the following  scheme  of the proof  by induction.   We start with the $1\times 1$ submatrix $L_1=\bigr(\lambda(x_1)\bigl)$ which already has a canonical form, so we set $x_1=y_1$.  Then we reduce $L_2$ to the 2-dimensional canonical form $\bigl(L_{\mathrm{can}}\bigr)_2$  by changing only one coordinate $y_2 \mapsto x_2$ and leaving all the others unchanged.    And so on, assuming that $L_k=\bigl(L_{\mathrm{can}}\bigr)_k$  (which means that the first canonical coordinates $x_1,\dots, x_k$ have been already constructed),  we reduce $L_{k+1}$ to the canonical form $\bigl(L_{\mathrm{can}}\bigr)_{k+1}$  by finding the next canonical coordinate $x_{k+1}$ in terms of $x_1,\dots, x_k$ and $y_{k+1}$. The process finishes in $n-1$ steps.

To complete the proof, we need to prove the following lemma justifying the above procedure.   To simply notation,  we assume that $k=n-1$. In other words, we treat the very last step which is absolutely similar to each previous one.

\begin{Lemma}
Suppose that a Nijenhuis operator  $L$ has been reduced to the  form
$$
L = \begin{pmatrix}
\, \boxed { \begin{matrix}   \\  \! \bigl(L_{\mathrm{can}}\bigr)_{n-1} \! \\  \\  \end{matrix}} & \begin{matrix} \!\! 0 \\ \!\! \vdots \\ \!\!  0 \end{matrix} \\
\, * \ \  \dots \ \  *   &  \!\!\! \lambda (x_1) 
\end{pmatrix}
$$
in some (adapted) coordinates $x_1,\dots, x_{n-1}, y_n$ (in other words, $L$ is ``almost canonical'' except for the last row which contains some functions denoted by $*$ and depending on all variables).  Then we may change the last coordinate only, i.e., set $x_n= f(x_1, \dots, x_{n-1}, y_n)$ leaving all the other coordinates unchanged in such a way that $L$ will take the desired canonical form \eqref{eq:nonconstJB}.
\end{Lemma}

\begin{proof}  It follows from \eqref{eq:nonconstJB} that the last coordinate $x_n$ is characterised by relation $L^* \dd x_n = \xi_n \dd x_1 + \dd x_{n-1} + \lambda \dd x_n$.  It means that the function $f$ can be found by solving the following partial differential equation:
\begin{equation}
\label{eq:laststep}
L_\lambda^*  \dd f = \dd x_{n-1}  - \lambda' (n-2) f \, \dd x_1.
\end{equation}

Let us first restrict this equation onto the leaves of the foliation $\mathcal F_1$ given by $\{ x_1= \mathrm{const}\}$. In other words, we think of $f$ as a function in $x_2,\dots, x_{n-1}, y_n$  dealing with  $x_1$ as a parameter.   Then on each leaf $\{ x_1 = \mathrm{const}\}$ we obtain a simpler equation
$L^*_\lambda \dd f = \dd x_{n-1}$  with $\lambda = \lambda(x_1)$ being a constant.  But this equation has been already analysed in the proof of Theorem \ref{thm:part1:9}  and we know that it admits a solution for each fixed $x_1$.   
Take an arbitrary solution $\tilde f =\tilde f (x_1; \, x_2,\dots, x_{n-1}, y_n)$ of this kind.   If we now consider $\tilde f$ as a function of all variables, then it satisfies \eqref{eq:laststep} modulo $\dd x_1$, i.e., we have  
$$
L^*_\lambda \dd \tilde f = \Bigl( \dd x_2  -   \lambda' (n-2) \tilde f \, \dd x_ 1 \Bigr) + g \, \dd x_1
$$
where $g$ is a certain smooth function.

If we set $\tilde x_n =\tilde f (x_1; \, x_2,\dots, x_{n-1}, y_n)$  ($x_1, x_2,\dots, x_{n-1}$ remain unchanged), then in this new coordinate system 
$x_1,  \dots, x_{n-1}, \tilde x_n$ the operator $L$ will coincide with \eqref{eq:nonconstJB} except for one single element, namely 
\begin{equation}
\label{eq:almostcan}
L = L_{\mathrm{can}} + P, \quad  \mbox{where }  P =\begin{pmatrix}
0 & 0 & \dots & 0 \\
\vdots & \vdots & \ddots & \vdots\\
0  & 0 & \dots & 0 \\
 g & 0 & \dots & 0
\end{pmatrix}
\end{equation}
Using the fact that $L$ and $L_{\mathrm{can}}$ are both  Nijenhuis,  we can verify by a straightforward computation that $g$ depends only on $x_{1}$ and $x_2$. 

We now consider equation \eqref{eq:laststep} again with $L$ given by \eqref{eq:almostcan} in coordinates $x_1, \dots, x_{n-1}, \tilde x_n$:  
$$
L^*_\lambda \dd  f = \dd x_{n-1}    - \lambda' (n-2) f  \dd x_n 
$$

We will be looking for $f$ in the form $f = \tilde x_n +  h (x_1, x_2)$.  Substitution gives
$$
L^*_\lambda \dd \bigl(\tilde x_n +  h \bigr) = \dd x_{n-1} - \lambda' (n-2) \bigl(\tilde x_n +  h \bigr) \dd x_1,
$$
or, in more detail:
$$
\dd x_{n-1} -  \lambda' (n-2) \tilde x_n \dd x_1 + g \dd x_1 + \frac{\partial h}{\partial x_{2}} \dd x_1  = \dd x_{n-1} - \lambda' (n-2) \bigl(\tilde x_n +  h \bigr) \dd x_1,
$$
which finally reduces to
$$
   \frac{\partial h}{\partial x_{2}}   = - g - \lambda' (n-2)   h,
$$
with $g=g(x_1, x_2)$ and $h=h(x_1,x_2)$. This equation obviously has a solution for any $g$.  This completes the proof.
\end{proof}

As explained above, this lemma immediately leads to the desired result. \end{proof}

\begin{Remark}
\label{rem:canformgeneral}
{\rm
If a Nijenhuis operator $L$ is similar  (over $\mathbb C$) to a pair of $n\times n$ Jordan blocks with complex conjugate eigenvalues $\alpha \pm \ii \beta$,  $\beta\ne 0$,  then its canonical form can be easily obtained by introducing  the canonical complex structure $J$ on $M$ and interpreting $L$ as a holomorphic Nijenhuis operator on the $n$-dimensional complex manifold $(M,J)$ as explained in Section \ref{noreal} (cf. Remark \ref{rem:fromrealtocompl}).  After this we can literally repeat the above construction with real variables $x_i$ replaced by complex variables $z_i$.  The final conclusion is the same as that of Theorem \ref{thm:part1:10} with straightforward modifications: $L_{\mathrm{can}}$ given by \eqref{eq:nonconstJB} is now a complex matrix,   $x_i$ should be replaced by $z_i$ and the function $\lambda (z_1)$ is holomorphic.  
}\end{Remark}

\begin{Remark}
\label{rem:canformgenpoints}{\rm
By the 
splitting theorem, in order to obtain a local description of a Nijenhuis operator in a neighborhood of an algebraically generic point, it is sufficient to do this  under the assumption  that the  operator has one real, or two complex conjugate eigenvalues (Theorem \ref{th:bols3.2}).  If in addition we assume that a Nijenhuis operator $L$ is $\mathrm{gl}$-regular, i.e.,  each eigenvalue corresponds to exactly one Jordan block, then  Theorems \ref{thm:part1:5} and \ref{thm:part1:10}  and Remark \ref{rem:canformgeneral},  provide this description.  Namely,  if $L$ is $\mathrm{gl}$-regular and algebraically generic, then locally $L$ splits into direct sum of blocks of  4 types:  
\begin{itemize}
\item  trivial $1\times 1$  Jordan block corresponding to a real eigenvalue $\lambda$:   $\bigl( \lambda (u)\bigr)$ where $\lambda(u)$ is a smooth function of $u$;
\item  trivial $2\times 2$  Jordan block corresponding to a pair of complex conjugate eigenvalues $\rho=a + \ii b$, $\bar \rho = a - \ii b$, $b\ne 0$:
$\begin{pmatrix}
a(x,y) & \!\!\! - b(x,y) \\
b(x,y) &  a(x,y)
\end{pmatrix}$,
where $\rho(x,y)=a(x,y) + \iii b(x,y)$ is a holomorphic function in  $z=x + \ii y$;
\item non-trivial Jordan block given by \eqref{eq:nonconstJB} corresponding to a real eigenvalue;
\item non-trivail complex Jordan block corresponding to a pair of complex conjugate eigenvalues obtained from \eqref{eq:nonconstJB} as explained in Remark \ref{rem:canformgeneral}.
\end{itemize}
}\end{Remark}

\section{Linearisation and left-symmetric algebras}\label{sect:5}

\subsection{Singular points of scalar type and linearisation}

In the previous section we studied the local structure and normal forms of Nijenhuis operators $L$ at those points which satisfy some kind of regularity conditions  (either algebraic regularity, or algebraic genericity, or differential non-degeneracy).  For instance,  Theorem \ref{thm:part1:5}  describes the simplest situation when the eigenvalues of $L$ are all distinct.  This section is devoted to the opposite case:  we will assume that all the eigenvalues of $L$ at a given point $p_0 \in M$ coincide and, even more than that,  we will assume that $L(p_0) = \lambda\cdot\Id$, $\lambda\in \R$, that is, 
$p_0$ is  a (singular) point of  {\it scalar type}   (Definition \ref{def:scalar}).

Since for any Nijenhuis operator $L$ and constant $\lambda\in\R$,  the operator $L - \lambda\cdot\Id$ is still Nijenhuis, then without loss of generality we may assume that at a point $p_0\in M$ of scalar type we have  $L(p_0) = 0$.  Under this assumption, we can apply a natural linearisation procedure for $L(x)$ at the point $p_0$ by expanding $L(x)$ into Taylor series in some local coordinate system $x^1, \dots, x^n$  (centred at $p_0$)  and taking the linear part\footnote{The same procedure is used for linearisation of a Poisson structure at a singular point \cite{weinstein}.}:
$$
L(x) = 0 + L_1 (x) +  L_2 (x) + L_3(x) + \dots 
$$
where the entries of $L_k(x)$ are homogeneous polynomials in $x^1, \dots, x^n$  of degree $k$ and 
$$
L_1  =  L_{\mathrm{lin}}  \quad   \mbox{with} \quad  (L_{\mathrm{lin}} )^i_j(x) = \sum_{k=1}^n  l^i_{j,k} x^k, \mbox{ where } l^i_{j,k} = \frac{\partial L^i_j}{\partial x^k}(p_0).
$$

Notice that the linear part $L_1=L_{\mathrm{lin}}$ is itself a Nijenhuis operator  that can be naturally understood as the linearisation of $L$ at the point $p_0\in M$ (indeed it is straightforward to check that its components $(L_{\mathrm{lin}} )^i_j(x)$ satisfy the Nijenhuis relations  $\bigl(\mathcal N_{L_{\mathrm{lin}}} \bigr)^i_{jk}=0$,  see Definition \ref{def:Nij2}).  In a more conceptual way,  the linearisation $L_{\mathrm{lin}}$ of $L$ should be considered as a $(1,1)$-tensor field on the tangent space $T_{p_0}  M$. 

\begin{Definition}
Let $L(p_0)=0$ and $\eta \in T_{p_0} M$. {\rm The {\it linearisation} of $L$ at $p_0\in M$ is defined as the operator field
$$  
L_{\mathrm{lin}} (\eta) : T_\eta (T_{p_0} M) \to T_\eta (T_{p_0} M)
$$
defined by the following formula:
\begin{equation}
\label{eq:invlinearisation}
\xi \mapsto \big[\,\widetilde \eta,  L \widetilde\xi\,\big] (p_0)  \quad \mbox{for any $\xi \in T_\eta (T_{p_0} M)\simeq T_{p_0} M$},
\end{equation}
where $\widetilde \xi$ and $\widetilde\eta$ are two arbitrary vector fields on $M$ such that $\widetilde \xi (p_0)=\xi$ and $\widetilde\eta(p_0)=\eta$.
}\end{Definition}

The following proposition is obvious.

\begin{Proposition}
\begin{enumerate}
\item[$(i)$]  The operator \eqref{eq:invlinearisation} is well defined, i.e., does not depend on the extensions $\widetilde \xi$ and $\widetilde \eta$ of the tangent vectors $\xi$ and $\eta$.

\item[$(ii)$] $L_{\mathrm{lin}} (\eta)$ is a Nijenhuis operator on the vector space $T_{p_0} M$  whose components, in any Cartesian coordinate system, are linear functions, i.e.,
$(L_{\mathrm{lin}} )^i_j(\eta) = \sum_{k=1}^n  l^i_{j,k} \eta^k$. 
Moreover,  if $x^1, \dots, x^n$ is a local coordinate system on $M$, then in the corresponding Cartesian coordinate system on $T_{p_0} M$ we have
$l^i_{j,k} =   \frac{\partial L^i_j}{\partial x^k}(p_0)$.
\end{enumerate}
\end{Proposition}

The same object can be defined in an algebraic way by interpreting  $l^i_{j,k}$  as a structure tensor of a certain $n$-dimensional algebra $(\mathfrak a_L, *)$ defined on the tangent space $T_{p_0} M$  (in the basis $e_1=\partial_{x^1}, \dots, e_n=\partial_{x^n}$).  In other words,  for any $\xi, \eta \in T_{p_0}$ we set by definition:
\begin{equation}
\xi * \eta = L_{\mathrm{lin}}(\eta)\xi =  \sum_i \left(\sum_{j,k}   l^i_{j,k} \xi^j \eta^k \right) e_i, \qquad l^i_{j,k} =   \frac{\partial L^i_j}{\partial x^k}(p_0).
\end{equation}

Since $L_{\mathrm{lin}}$ is a Nijenhuis operator,  this algebra is expected to be rather special\footnote{Recall that in the case of Poisson tensors, the linearisation procedure also leads to an algebra with very special properties,  namely, a Lie algebra.}.  What can we say about it?  The answer can be found in literature (see, e.g., \cite{winterhalder}) although it does not seem to be widely known, at least amongst differential geometers.  

\begin{Definition}{\rm
An algebra $(\mathfrak a, *)$ is called {\it left-symmetric}  if the following identity holds:
\begin{equation}
\label{eq:LSA}
\xi * (\eta*\zeta) - (\xi*\eta)* \zeta = \eta*(\xi*\zeta) - (\eta*\xi)*\zeta,  \quad \mbox{for all } \xi,\eta,\zeta\in \mathfrak a.
\end{equation}
}\end{Definition}

The definition is due to Vinberg \cite{vinberg};  recent papers on left-symmetric algebras include \cite{burde,burde2}.

Notice that  \eqref{eq:LSA}  can be rewritten as $\mathsf L_\xi \mathsf L_\eta -  \mathsf L_\eta \mathsf L_\xi = \mathsf L_{[\xi, \eta]}$  where $\mathsf L_\xi :  \mathfrak a \to \mathfrak a$ denotes the left multiplication by $\xi$, i.e.  $\mathsf L_\xi \zeta = \xi* \zeta$. The latter relation implies that the commutator $[\xi, \eta] = \xi * \eta - \eta * \xi$ defines the structure of a Lie algebra  on $\mathfrak a$ (called the Lie algebra associated with $\mathfrak a$).

The relation between (linear) Nijenhuis operators and left-symmetric Lie algebras is very natural.

\begin{Proposition}
\label{prop:1.?} 
Consider a (1,1)-tensor field $L$ which is defined on a vector space $T\simeq\R^n$ and such that all of its components are linear functions in Cartesian coordinates $x^1, \dots, x^n$,   that is,
\begin{equation}
\label{eq:lN}
L^i_j (x) = \sum_k l^i_{j,k} x^k.
\end{equation}

Then $L(x)$  is a Nijenhuis operator if and only if $l^i_{j,k}$ are structure constants of a left-symmetric algebra.  We will denote this left-symmetric algebra by $(\mathfrak a_L,*)$.
\end{Proposition}

The proof of Proposition \ref{prop:1.?} is straightforward and can be found e.g. in \cite{winterhalder},  we also include it in Part II \cite{Part3}  of our series of papers on Nijenhuis geometry.

Thus, there is a natural bijection between linear Nijenhuis tensors and left-symmetric algebras.  In this view,  sometimes by the linearisation of a (non-linear) Nijenhuis tensor $L$ at a singular point $p_0\in M$ of scalar type, we will understand  the left-symmetric algebra $\mathfrak a_L$ defined on  $T_{p_0}M$. 
We will refer to it as the \emph{isotropy left-symmetric algebra}. 

The linearisation problem for a Nijenhuis operator $L$ in a neighbourhood of a scalar-type singular point $p_0$ can be stated as follows:  under which conditions $L$ is (locally) equivalent to its linearisation, that is, to the linear Nijenhuis operator $L_{\mathrm{lin}}$ associated with $\mathfrak a_L$?  In this case we say that $L$  is {\it linearisable} at the singular point $p\in M$.
Sometimes linearisability may follow from the structure of $\mathfrak a_L$.   The next definition is borrowed from Poisson Geometry.  The only modification is that Lie algebras are replaced with left-symmetric algebras and Poisson tensors are replaced with NIjenhuis operators.  

\begin{Definition}\label{def:nondegLSA}
{\rm
A left-symmetric algebra $\mathfrak a$ is called {\it non-degenerate}  if any Nijenhuis operator $L$,  whose isotropy left-symmetric algebra $\mathfrak a_L$ at a singular point $p_0$ is isomorphic to $\mathfrak a$, is linearisable at this point. 
}\end{Definition}

Equivalently,  non-degeneracy of $\mathfrak a$ means that the corresponding linear Nijenhuis operator is $C^2$-stable  (at the origin).

Notice that speaking of linearisability we have to distinguish at least three different cases:  smooth, real analytic and formal.  In this view,   non-degeneracy of a left-symmetric algebra $\mathfrak a$ can be understood in three different settings:  smooth, real analytic and formal.   The following example is treated in detail in Part II of our work \cite{Part3}.

\begin{Ex}{\rm
Consider the two-dimensional left symmetric algebra $\mathfrak b_{1,\alpha}$ with a basis $e_1$, $e_2$ and relations 
$$
e_1*e_1 = 0, \quad e_1 * e_2 = 0, \quad   e_2 * e_1 = e_1, \quad e_2 * e_2 =\alpha e_2.
$$

The corresponding linear Nijenhuis operator is
$$
L_{\mathfrak b_{1,\alpha}} = \begin{pmatrix}  0 & x \\ 0 & \alpha y \end{pmatrix},  \quad \alpha\in\R.
$$

If $\alpha <0$ and is irrational, then this left-symmetric algebra is degenerate in the smooth sense. To show this one can consider the following {\it smooth} perturbation of $L_{\mathfrak b_{1,\alpha}}$:
$$
L =  \begin{pmatrix}  0 & x \\ 0 & \alpha y \end{pmatrix} + \begin{pmatrix}   h (x,y) & g (x,y) \\ 0 & 0 
\end{pmatrix}
$$
where $h$ and $g$ are both flat functions defined by:
$$
h(x,y) = \begin{cases}
\exp\left( - \frac{1}{x^2 (y^2)^{-\frac{1}{\alpha}}}  \right),   & xy\ne 0,\\
0, & xy = 0.
\end{cases} \quad \mbox{and} \quad
g(x,y) = \begin{cases}
h(x,y) \frac{h_y (x,y)}{h_x(x,y)},   & xy\ne 0,\\
0, & xy = 0.
\end{cases}
$$
They satisfy two properties:   1)  $h(x,y)$ is a smooth first integral for the vector field $(x, \alpha y)$  (defined by the second column of 
$L_{\mathfrak b_{1,\alpha}}$ and 2) the relation  $h_x g = h_y h$ holds. These conditions guarantee that  $L$ is Nijenhuis. 

It is obvious, however, that $L$ is not equivalent to $L_{\mathfrak b_{1,\alpha}}$ since the latter has zero eigenvalues at each point but $L$ does not.  

This trick with a flat perturbation does not work in the real analytic case (in particular, the vector field $(x, \alpha y)$ admits no real analytic first integrals) and it can be shown  (see Part 3 \cite{Part3} for details) that the algebra $\mathfrak b_{1,\alpha}$ is analytically non-degenerate for almost all negative irrational $\alpha$'s. 

}\end{Ex}

In the next section, we prove the non-degeneracy for the simplest and most natural left-symmetric Lie algebra, namely, the diagonal one. In general,  the classification problem for left-symmetric algebras is completely open and, to the best of our knowledge,  has not been even touched.

\subsection{Non-degeneracy of the diagonal left-symmetric algebra}\label{sect:5.2}

\begin{Theorem}
\label{thm:lin}
Let $\mathfrak d$ be a non-trivial one-dimensional LSA and $\mathfrak d^n =\oplus_n \mathfrak d$ be the direct sum of $n$ copies of $\mathfrak d$. Then $\mathfrak d^n$ is non-degenerate both in formal and real analytic sense. Equivalently,  consider the linear Nijenhuis operator of the form 
$$
L_{\mathrm{lin}}(x)=\mathrm{diag}(x_1,x_2,\dots, x_n)
$$
and a formal (resp., real analytic)  perturbation of $L_{\mathrm{lin}}$ of the form
\begin{equation}
\label{eq:perturb}
L(x) = L_{\mathrm{lin}}(x) + L_2(x) +L_3(x) +\dots ,
\end{equation}
where the entries of $L_k(x)$ are homogeneous polynomial of degree $k\ge 2$ in $x$ and $L$ is a formal (resp. real analytic) Nijenhuis operator. Then $L(x)$ is formally (resp., real analytically)  linearisable, i.e., there exists a formal  (resp. real analytic) change of variables 
\begin{equation}
\label{eq:coord_change}
y_i = x_i + f_{i,2}(x) + f_{i,3} (x) + \dots
\end{equation}
with $f_{i,k}$ being homogeneous polynomial of degree $k\ge 2$ such that in new coordinates the operator $L$ takes the form
$$
L_{\mathrm{lin}}(y)=\mathrm{diag}(y_1,y_2,\dots, y_n),
$$    
or, more specifically,
\begin{equation}
\label{eq:similarity}
\left( \frac{\partial y}{\partial x}\right)  \bigl(L_{\mathrm{lin}}(x)+ L_2(x) +L_3(x) +\dots \bigr) \left(\frac{\partial y}{\partial x}\right) ^{-1}= L_{\mathrm{lin}}(y),
\end{equation}
where $\left( \dfrac{\partial y}{\partial x}\right)$ is the Jacobi matrix of the coordinate change \eqref{eq:coord_change}.
\end{Theorem} 

\begin{Remark} {\rm In dimension two, the statement of Theorem \ref{thm:lin} remains true in the smooth category \cite{Part3}.  However, we do not know if it is still the case for $k\ge 3$.   
}\end{Remark}

\begin{Remark} {\rm 
We  do not know if the direct sum of two non-degenerate left-symmetric algebras is still non-degenerate and therefore we cannot 
use this obvious idea in our proof.  This general property might be true but we do not see any elementary explanation of it.
}\end{Remark}

\begin{Remark} {\rm 
The geometric meaning of the new coordinates $y_1, \dots, y_n$ is very simple. These are just the eigenvalues of the perturbed operator $L$.  However, the argument that ``under small perturbations the eigenvalues change in a good way'' does not work in this case at all.  The point is that generic (i.e., non-Nijenhuis) perturbations of $L_{\mathrm{lin}}(x)$ immediately destroy all good properties of the eigenvalues,   typically they become multivalued and by no means smooth.  Theorem \ref{thm:lin}  is basically equivalent to the fact that the eigenvalues of $L_{\mathrm{lin}}(x)$ remain smooth under Nijenhuis perturbations.
}\end{Remark}

\begin{proof} We first prove {\it formal} non-degeneracy of $\mathfrak d^n$, or equivalently,  {\it formal} linearisability of $L(x)$,   by killing each term $L_k(x)$ of the perturbation step by step. In other words, we will use the traditional approach based on the following sufficient condition for formal linearisations.  As before, throughout the proof $L_k(x)$ denotes a matrix whose entries are homogeneous polynomials in $x$ of degree $k\ge 2$.
 
We first of all notice that if $L(x) = L_{\mathrm{lin}}(x) + L_k(x) + \dots$ is a Nijenhuis operator, then the  Fr\"olicher--Nijenhuis bracket of $L_{\mathrm{lin}}(x)$ and $L_k(x)$ vanishes  (see Definition \ref{def:FNbracket}).  In this case we will say that $L_{\mathrm{lin}}$  and $L_k$ are {\it compatible}. 
  
\begin{Proposition}
\label{prop:forLinear2} 
 Assume that for any $k$ and for any $L_k$ compatible with $L_{\mathrm{lin}}$ there exists a change of variables 
$$
y_i = x_i + f_{i,k}(x), \quad\mbox{where $f_{i,k}(x)$ is a homogeneaous polinomial in $x$ of degree $k$,}
$$  
that transforms  $L(x)  =  L_{\mathrm{lin}}(x) + L_k(x)$ to the form $L(y)  =  L_{\mathrm{lin}}(y) + L_{k+1}(y)+\dots$ (i.e., kills terms of order $k$). 
Then any Nijenhuis perturbation of $L_{\mathrm{lin}}$ of the form \eqref{eq:perturb} is formally linearisable.
\end{Proposition}

Notice that this statement holds true for any linear NIjenhuis operator, not necessarily the diagonal one $L_{\mathrm{lin}}(x)=\mathrm{diag}(x_1,x_2,\dots, x_n)$.

This sufficient condition for formal linearisability seems to be standard and we omit the proof.  We just want to show that it is fulfilled in our case. Namely, we are going to prove the following

\begin{Proposition}
\label{prop:forLinear1}
Let $R(x)$ be an operator whose entries are homogeneous polynomials of degree $k$ and such that $[L_{\mathrm{lin}}(x), R(x)]_{\mathrm{FN}}=0$. Then there exists a change of variables 
$$
y_i = x_i + f_i(x)
$$
with $f_i$ being homogeneous polynomials of degree $k$ such that the following relation holds:
\begin{equation}
\label{eq:fordn}
\bigl( \Id + J\bigr)^{-1} \Bigl( L_{\mathrm{lin}}(y) \  + \ (\mbox{terms of degree $\ge k+1$} )\Bigr)  \bigl( \Id + J \bigr) =  L_{\mathrm{lin}}(x)+ R(x),
\end{equation}
where $J = \left(\frac{\partial f}{\partial x}\right)$ and $F(x) =\mathrm{diag}\bigl( f_1(x),\dots, f_n(x)\bigl)$.
\end{Proposition}

\begin{proof}
By equating $k$-order terms in \eqref{eq:fordn},  we obtain
\begin{equation}
\label{eq:2}
R(x) - F(x)= [L_{\mathrm{lin}}(x) , J ] 
\end{equation}
  This system is easy to analyse.

\begin{Lemma}
The change of variables $y_i=x_i+f_i(x)$ with required conditions \eqref{eq:2} exists if and only if the components of $R(x)$ satisfy the relation:
\begin{equation}
\label{eq:3}
R^i_{j} (x) =   (x_i - x_j)  \frac{\partial R^i_{i}}{\partial x_j}.
\end{equation}
If \eqref{eq:3} holds,  this change of variables is unique and takes the form
$$
y_i = x_i + R^i_{i}(x).
$$
\end{Lemma}

\begin{proof}
Assume that \eqref{eq:2} holds. Taking into account the fact that the diagonal elements of  $[ L_{\mathrm{lin}}(x) , J]$ vanish, we come to the condition that  $R^i_{i}(x)$ (diagonal elements of the perturbation) must coincide with $f_1, \dots, f_n$.  In other words, if for a given perturbation $R(x)$  the desired change of variables exists then such a change is unique and is given by the diagonal elements of the perturbation, i.e.
$$
y_i = x_i + R^i_{i}(x).
$$
On the other hand, off-diagonal elements $R^i_{j}$ of $R(x)$ must coincide with the off-diagonal elements of  the (matrix) commutator  $[L_{\mathrm{lin}}(x) , J ]=
L_{\mathrm{lin}}(x) J - J L_{\mathrm{lin}}(x)$, which are of the form $(x_i-x_j) J^i_{j}=(x_i-x_j) \frac{\partial f_i}{\partial x_j}=   (x_i - x_j)  \frac{\partial R^i_{i}}{\partial x_j}$ as stated. Conversely, if \eqref{eq:3} holds, it is easily seen that the change of variable  $y_i = x_i + R^i_{i}(x)$ satisfies \eqref{eq:2}. \end{proof}

It remains to check that condition \eqref{eq:3} follows from to the fact the Fr\"olicher--Nijenhuis bracket of $L_{\mathrm{lin}}(x)$ and $R(x)$ vanishes.

\begin{Lemma}
Let $L_{\mathrm{lin}}$ and $R(x)$ be compatible, i.e.,  $[L_{\mathrm{lin}}, R(x)]_{\mathrm{FN}}=0$, then \eqref{eq:3} holds. 
\end{Lemma}

\begin{proof} Vanishing the Fr\"olicher--Nijenhuis bracket of $L_{\mathrm{lin}}$ and $R$ amounts to the following relations (for all $i\ne j$):
$$
L_{\mathrm{lin}} [R^\alpha_j \partial_\alpha, \partial_i] + L_{\mathrm{lin}} [ \partial_j, R^\alpha_i \partial_\alpha	]  + R[x_i \partial_i , \partial_j] + R[\partial_i , x_j\partial_j]
- [R^\alpha_j \partial_\alpha, x_i\partial_i] -  [ x_j\partial_j, R^\alpha_i \partial_\alpha	] =0,
$$
with summation over $\alpha$ (but not over $i$ and $j$).
A straightforward computation gives:
$$
-\sum_\alpha x_\alpha\frac{\partial R^\alpha_j}{\partial x_i} \partial_\alpha + \sum_\alpha x_\alpha \frac{\partial R^\alpha_i}{\partial x_j} \partial_\alpha + 0 + 0 +
\sum_\alpha x_i \frac{\partial R^\alpha_j}{\partial x_i} \partial_\alpha - R^i_j \partial_i  - 
\sum_\alpha x_j \frac{\partial R^\alpha_i}{\partial x_j} \partial_\alpha + R^j_i \partial_j =0
$$

Taking the $i$-th component of this vector identity, we get:
$$
-x_i\frac{\partial R^i_j}{\partial x_i}  +  x_i \frac{\partial R^i_i}{\partial x_j}  +
x_i \frac{\partial R^i_j}{\partial x_i}  - R^i_j    
- x_j \frac{\partial R^i_i}{\partial x_j} =     (x_i-x_j) \frac{\partial R^i_i}{\partial x_j}  
  - R^i_j  = 0,
$$
which coincides with \eqref{eq:3}.  \end{proof}

These two lemmas complete the proof of Proposition \ref{prop:forLinear1}. \end{proof}

Taking into account the sufficient condition of formal linearisability given by Proposition \ref{prop:forLinear2}, we conclude that $L(x)$ is formally linearisable.   

In the real analytic case, the proof easily follows from the Artin theorem \cite{Artin} which states that if an analytic equation admits a formal solution, then it also admits an analytic solution which is closed to the formal one in natural sense.  In particular, this theorem implies the following fact:  if a polynomial equation
$$
P(y) = y^n + \sigma_{n-1} (x) y^{n-1} + \sigma_{n-2} (x) y^{n-2} + \dots + \sigma_{1} (x) y + \sigma_{0} (x) =0,  \quad x\in \R^n, y\in\R,
$$ 
admits $n$ distinct formal solutions 
$y_k=h_k(x)$  (in the neighbourhood of $x=0$), then these solutions are automatically real analytic in a small neighbourhood of $x=0$.  

Formal linearisability of $L(x)$  means that the new coordinates 
$y_k$ are {\it formal} roots of the characteristic polynomial of the operator $L(x)$  (see formula \eqref{eq:similarity}).  According to the above particular case of the Artin theorem,   $y_k(x)$ are, in fact, real analytic  (in a small neighbourhood of $x=0$),  as required.  This completes the proof in the real analytic case.  \end{proof}

\section{Applications}\label{sect:6}

\subsection{Complex eigenvalues of Nijenhuis operators on closed conected manifolds are always constant}

\begin{Theorem}
\label{appl1}
Let  $L$ be a Nijenhuis operator on a  closed connected  manifold  $M$  with a non-real eigenvalue $\lambda\in \mathbb C\setminus \R$ at least at one point.
Then  this number $\lambda$ is an eigenvalue of $L$ with the same algebraic multiplicity 
at every point of $M$.  
\end{Theorem}
\begin{proof}

Let $k(x)$ be the number of non-real (perhaps, repeated) eigenvalues of $L$ at $x\in M$, and $\lambda_1(x),\dots ,\lambda_{k(x)}(x)$ denote all the corresponding non-real eigenvalues of $L(x)$. Next, define 
$$ 
\mu (x) := \max_{i= 1,\dots ,k(x)} \Im (\lambda_i(x)) \quad\mbox{and}\quad \mu_0:= \max_{x\in M} \mu(x),
$$
where $\Im(\cdot)$ denotes the imaginary part of a complex number.
At the points where $k(x)=0$ we set $\mu(x)=0$. 
Let $M_0\subset M$ be the set of those points where the maximum is achieved:
$$
M_0:= \{x\in M \mid \mu(x)= \mu_0\}.
$$ 
Clearly, $\mu(x)$  is a continuous function on $M$, so 
the compactness of the manifold implies that $M_0$ is not empty and is compact. Next, 
for each point $x\in M_0$, let $\widetilde k(x)\in\mathbb N$ be the number of those eigenvalues of $L(x)$  whose imaginary part equals $\mu_0$.  Denote such eigenvalues by $\lambda_1(x),...,\lambda_{\widetilde k(x)}(x)$ and their algebraic multiplicities by $m_1(x),...,m_{\widetilde k(x)}(x)$.  
Next, consider  the function 
$$
m:M_0\to \mathbb{N},  
 \quad M_0\ni x\mapsto   \left( \max_{i=1,...,\widetilde k(x)} m_i(x) \right), 
 $$
  denote by $m_0$ its maximum and  by $\widetilde M_0 \subset M_0$ the set of points where the maximum is attained. 
  By construction, at every point  $x$  of $M$ and for every eigenvalue $\lambda$ of $L(x)$, 
at least one of the following conditions holds: 
 \begin{enumerate} 
 \item $\Im(\lambda)< \mu_0$ or 
 \item the algebraic  multiplicity of $\lambda$ is not greater than $m_0$. 
 \end{enumerate}   
 Clearly, the set $\widetilde M_0$ is closed. Let us now show that it is open. Take arbitrary   $p\in \widetilde M_0$, and let 
 $\lambda_0$  be an  eigenvalue of $L(p)$ of algebraic multiplicity $m_0$ with $\Im (\lambda_0)= \mu_0$. 
 
  According to Theorems \ref{thm1} and \ref{thm:comlexcase}, there exists a connected neighbourhood $U=U(p)\subset M$ with a coordinate system  $z_1=x_1+ \ii y_1, \dots , z_{m_0}=x_{m_0}+\ii y_{m_0}$, $x_{2m_0+1},\dots,x_{n}$  and a  
  decomposition of the characteristic   polynomial $\chi_L(t)$ in the product  of three   monic polynomials 
 $$ 
 \chi_L(t)= P(t) \cdot \bar P(t) \cdot Q(t) 
 $$
satisfying the following  conditions hold: 
\begin{itemize} 
\item   the coefficients of the polynomial $
   P(t)= t^{m_0}+ a_{1}  t^{m_0-1} + \dots + a_{m_0}$  are independent of $x_{2m_0+1}, \dots ,x_{n}$ and are smooth complex-valued holomorphic  functions of the variables $z_1, \dots , z_m$;
   \item the coefficients of the polynomial $\bar P(t)$ are complex-conjugate to that of $P(t)$; 
\item    the coefficients of the polynomial $Q(t)= t^{n-2m_0}+ b_{1}  t^{n-2m_0-1} + \dots + b_{n-2m_0}$ are smooth real functions of the variables $x_{2m_0+1}, \dots ,x_{n}$;
\item at the point $p\in M$,  the polynomial $P(t)$ is equal to $(t- \lambda_0)^{m_0}$.
  \end{itemize} 
  
  By assumptions, the function $-\Im(a_{1})$ is the sum of imaginary parts of some $m_0$  eigenvalues of $L$  and therefore is at most $m_0\cdot \mu_0$, which implies that it 
   attains its maximum at $p$. Since $a_{1}$ is holomorphic, it is constant on $U$ by the maximum principle. Then the imaginary parts of the  roots of 
   $P(t)$ 
are constants on $U$.  Since the roots of the polynomial $P(t)$ (with holomorphic coefficients) are also holomorphic on an open everywhere dense subset of $U$, this condition implies that these roots themselves are constant on $U$. Hence, $U\subseteq \widetilde M_0$ and therefore $\widetilde M_0$ is open.  

Finally, since $\widetilde M_0$ is non-empty, 
closed and open, it coincides with the whole $M$, so that at every point of $M$ the eigenvalue $\lambda_0$ has multiplicity $m_0$.  Next, consider the operator 
$L':= (L- \lambda_0 \cdot \Id)(L-\bar \lambda_0\cdot \Id ) $. Its Nijenhuis tensor   is zero, and at every point the number of non-real eigenvalues of $L'$ is one less than that of $L$.   If all eigenvalues of $L'$ are real, we are done, since the only non-real eigenvalues  of $L$ are constants $\lambda_0$ and $\bar \lambda_0$, both of algebraic  multiplicity $m_0$. 
Otherwise,   replacing $L$ by $L'$ in all previous considerations, we obtain that one of 
 its complex eigenvalues (say, $\lambda_0'$) is  constant and has constant algebraic multiplicity  at all points   of $M$. Then, we consider $L''=  (L'- \lambda'_0\cdot \Id)(L'-\bar \lambda'_0\cdot \Id) $ and so on; in finitely many steps we come to a Nijenhuis operator with only real eigenvalues. 
Theorem    \ref{appl1} 
  is proved.  \end{proof}

It easily follows from Theorem  \ref{thm:part1:7} that near a differentially non-degenerate singular point there always exist points with non-constant complex eigenvalues. Hence we immediately obtain

\begin{Corollary}\label{cor:Matv1}
 A Nijenhuis operator $L$ on a closed manifold cannot have differentially non-degenerate singular points.
\end{Corollary}

The next corollary shows that the topology of a manifold carrying a Nijenhuis operator $L$ may ``affect'' the spectrum of $L$.
  
  \begin{Corollary}\label{cor:Matv2}
   The eigenvalues of a Nijenhuis operator on the 4-dimensional sphere $S^4$  are all real.
  \end{Corollary} 

\begin{proof} 
Let $L$ be a Nijenhuis operator on $S^4$ with a complex eigenvalue $\lambda$ at a point $x$ 
  which without loss of generality can be assumed to be equal to $\ii= \sqrt{-1}$ (otherwise replace $L$ by
   its  appropriate linear combination with $\mathrm{Id}$). Then, by Theorem \ref{appl1} at every point of $S^4$ the numbers 
  $\ii$ and $-\ii$ are eigenvalues of $L$.

 If $ \ii$ has algebraic multiplicity $2$ at $x$,  then $\ii$ and $-\ii$ have  algebraic multiplicity $2$ at every point of $S^4$ and therefore $L$ is a Nijenhuis operator  with no real eigenvalues.   According to Theorem \ref{thm:comlexcase}, in this case $L$ induces a natural complex structure $J$ on $S^4$ which is impossible.

Let us consider the case when  $\ii $ has algebraic multiplicity  $1$, and    construct an almost complex structure on $S^4$.
 Consider any Riemannian metric  $g$  on $S^4$,  the 2-dimensional distribution $\mathcal D= \Ker\left(L^2 + \mathrm{Id} \right)$ and its orthogonal complement $\mathcal D^\perp$. 
 
 Next, define an almost complex structure $J$ as follows: we  choose  an orientation 4-form $\omega$  and 
 for the vector $v= u + u^\perp$, where $u\in \mathcal D$ and $u^\perp\in \mathcal D^\perp$,  we put $J(v)= L(u) +  R(u^\perp)$, where $R$ is the $\tfrac{\pi}{2}$-rotation in the 2-plane $\mathcal D^\perp$ in positive direction. More precisely, 
$R(u^\perp)$ lies in $\mathcal D^\perp$, is orthogonal to $u^\perp$, has the same length as $u^\perp$, and $\omega(u, L(u), u^\perp , R(u^\perp))\ge 0  $  for any nonzero $u$.  Clearly, $J^2=-\Id$, so it is an almost complex structure.
  
  Now, by the classical result of Steenrod \cite[41.20]{steenrod}, the 4-sphere does not admit any almost complex structure. This contradiction completes the proof. 
\end{proof}

\subsection{Geodesically equivalent metrics near  differentially \\ non-degenerate singular points} \label{sect:geod}

Recall that two pseudo-Riemannian metrics $g$ and $\bar g$ on $M$ are \emph{geodesically equivalent}, if any $g$-geodesic  is a $\bar g$-geodesic. In this definition we consider geodesics without preferred parameterization. Geodesically equivalent metrics is a classical topic and appeared already in the papers of E. Beltrami, T. Levi-Civita and H. Weyl;  it has a revival in the last decades  due to new methods coming from the theory of integrable systems and  parabolic geometry which led to solving a series of  named  and natural problems.

Let us recall, following \cite{BoMa2003},  the relation between geodesically equivalent metrics and Nijenhuis operators.  
 Consider the operator  $L=L(g,\bar g)$ defined by
  
\begin{equation}
\label{L}
L_j^i := { \left|\frac{\det(\bar g)}{\det(g)}\right|^{\frac{1}{n+1}}} \bar g^{ik}
 g_{kj}. 
\end{equation}

By  \cite[Theorem 1]{BoMa2003},   if the metrics $g$ and $\bar g$ are geodesically equivalent, 
then $L$ is a Nijenhuis operator.  Clearly, the pair $(g, L)$ contains as much information as the pair $(g, \bar g)$ so that the geodesic equivalence condition for $g$ and $\bar g$ can be written as an equation for $g$ and $L$.  The compact form of  this equation is  
\begin{equation} \label{b1} 
\{ H, F \} = 2 H \ell . \end{equation}

Here $\{\  , \  \}$ is the canonical Poisson brackets on $T^*M$, 
$H$ is the hamiltonian of the geodesic flow of $g$,  $H:= \tfrac{1}{2}  g^{ij}(x) p_ip_j$,  $F$ is the quadratic is momenta function given by $F= L^i_k g^{kj} p_i p_j$,  and
$\ell$ is the linear in momenta function corresponding to the differential of the trace of $L$, 
$
\ell = \sum_i \frac{\partial \trace(L)}{\partial x_i}p_i,
$ 
where $(x_1,\dots, x_n, \, p_1,\dots, p_n)$ are usual canonical coordinates on $T^*M$.  

If \eqref{b1} holds for  a $g$-selfadjoint Nijenhuis operator $L$, we say that $g$ and $L$ are {\it geodesically compatible}\footnote{Notice that $L$ defined by \eqref{L}  should, in addition, be non-degenerate, but  we may temporarily ignore this condition as \eqref{b1} is preserved under shifts $L \mapsto L + c\cdot \Id$, $c\in\R$.}.
It is natural to ask if a given Nijenhuis operator $L$ admits a geodesically compatible metric $g$ (not necessarily positive definite)? 
 
For diagonalisable Nijenhuis operators with $n$ distinct eigenvalues, the answer is positive: there always exists two geodesically equivalent metrics $g$ and $\bar g$ such that $L$ is given by \eqref{L}. Indeed, in this case and under the additional assumption that the eigenvalues are real,   there exists  a coordinate system  such that $L= \textrm{diag}(\lambda_1(x_1),...,\lambda_n(x_n))$, and the metrics  $g$ and $\bar g$ are as follows: 
\begin{eqnarray} 
g&=& \sum_{i=1}^n\left(\varepsilon_i\prod_{j=1; j \ne i}^n (\lambda_i- \lambda_j)\right)\ddd x_i^2 \label{eq:LC}\\ 
 \bar g&=& \sum_{i=1}^n\left(\frac{\varepsilon_i}{\lambda_i^2}\prod_{j=1; j \ne i}^n \tfrac{1}{\lambda_j}(\lambda_i- \lambda_j)\right)\ddd x_i^2, 
\end{eqnarray} 
where $\varepsilon_i\in \{-1,1\}$.  This result was  obtained by Levi-Civita  \cite{Levi-Civita} in 1896. Generalization of this result to the mixed case with real and complex eigenvalues (but still distinct) is also known, see e.g. \cite{BoMa2015} (special cases were known long before this paper).    

Next assume that a Nijenhuis operator $L$ is algebraically generic  (Definition \ref{def:alggen}).  This case is also known and the answer is as follows: 
{\it If $L$ admits a  geodesically compatible $g$, then the geometric multiplicity of any nonconstant eigenvalue of $L$ is one.} In particular,   
for a semisimple  Nijenhuis operator $L$,  a geodesically compatible metric $g$ exists
if and only if all eigenvalues of multiplicity greater than $2$ are constant.  The direction ``$\Longrightarrow$'' is proved in  \cite{BoMa2011}, while the direction ``$\Longleftarrow$'' follows from \cite{BoMa2015}.  

However, if we consider a Nijenhuis operator $L$ near a singular point $p\in M$, then existence problem for a projectively equivalent partner remains widely open. Only the Riemannian case under  the additional  assumption that the metrics are strictly-nonproportional is fully understood \cite{matvadd}.
 In almost all known global examples, such points do appear and, moreover,  turn out to be interesting from geometric viewpoint  (e.g., ombilic points on ellipsoids).  It is expected that that classification of such points may help to describe the topology of closed manifolds admitting projectively equivalent metrics and to prove the projective Lichnerowicz conjecture for all signatures.

The main results of this section,  Proposition \ref{prop:1} and Theorem  \ref{thm:ra} below,  concern  differentially non-degenerate singular points (Definition \ref{def:nondeg}).  In this case the answer is positive: there locally exists a metric geodesically compatible with $L$. Indeed, as  shown in Theorem \ref{thm:part1:7}, in this case there exists a coordinate system such that $L$ is given by
\begin{equation} \label{b2}
L = \begin{pmatrix}
x_1 & 1 &  & & \\
x_2 & 0 & 1 & & \\
\vdots & \vdots & \ddots & \ddots & \\
x_{n-1} &  0 & \dots & 0 & 1\\
x_n & 0  & \dots & 0 & 0 
\end{pmatrix},
\end{equation} 
and we have

\begin{Proposition}\label{prop:1} 
The metric $g$ whose  dual (=inverse) metric $\bigl(g^{ij}\bigr)= g^{-1}$  is  given by 
  \begin{equation} \label{metricg}
g^{-1}= 
\begin{pmatrix}    
0 & \cdots & 0 & 0 & \!\!\!\! -1\\ 
0 &  \cdots & 0  &\! -1 & x_1\\ 
\vdots & \iddots &\iddots & \iddots  &x_{{2}} \\ 
0&\!\!-1&x_{{1}}& \iddots &   \vdots    \\
-1&x_{{1}}&x_{{2}}  &\cdots & \!\!\! x_{{n-1}}
  \end{pmatrix}  
\end{equation}    
is geodesically compatible with $L$ given by \eqref{b2}.

\end{Proposition}

The proof is left to the reader: one needs to check by direct calculations that the operator  \eqref{b2} is $g$-selfadjoint and \eqref{b1}  is satisfied.

Recall that one metric $g$  geodesically compatible with  $ L$ allows us to construct infinitely many metrics $g_f$ geodesically compatible to $L$. This construction is essentially due to Topalov \cite{To}, its special case was known to Sinjukov \cite{Si1}, and  it is described in detail in \cite[\S 1.3]{BoMa2011}.  Given geodesically compatible pair $(g, L)$ and an analytic matrix function in the sense of Section \ref{analytic}, one constructs a metric $g_f$ geodesically compatible to the same $L$ by setting $g_f (\xi, \eta) = g\bigl( f(L)\xi, \eta\bigr)$ or shortly $g_f=g f(L)$.

The main result of this section is a description of all real analytic metrics compatible with $L$ near  a differentially non-degenerate singular point. 

\begin{Theorem} \label{thm:ra} Let $g'$ and $L$ be geodesically compatible and real analytic. Assume  that $L$ is  differentially non-degenerate at a point $p\in M$ and $L(p)$ is conjugate to the $n\times  n$ Jordan block with zero eigenvalue. Then there exists a real analytic  function $f: (-\varepsilon, \varepsilon) \to \R$ such that in some local coordinate system $x_1,\dots, x_n$ centered at $p\in M$, the operator $L$ is given by \eqref{b2}  and $g'=g_f$, where $g$ is given by \eqref{metricg}.   
\end{Theorem} 

Examples show that the assumption that the metric is real analytic  is important. 

\begin{proof}  
The existence of coordinates  such that  $L$ is given by \eqref{b2}  was proved in Theorem \ref{thm:part1:7}, moreover such a coordinate system is unique and real analytic.   We can assume therefore that we work on $U\subseteq \mathbb{R}^n$ with $(x_1,...,x_n)$ being Cartesian coordinates, $(0,...,0)\in U$ and $L$ is as in \eqref{b2}.  We need to show that $g'= g_f$ for a real analytic function $f$. Because of real analyticity,  it is  sufficient to prove the last statement in any open subset $V \subset U$.  Let us choose such a subset $V$. 

Consider  the mapping 
\begin{equation}
\label{eq:phi}  
\phi:\mathbb{R}^n\to \mathbb{R}^n, \  
\phi(y)= \bigl(\sigma_1(y),\sigma_2(y),...,\sigma_n(y)\bigr),
\end{equation}
 where $\sigma_k$ is the elementary symmetric polynomial of  degree $k$.   By implicit function theorem, $\phi$ is a local  diffeomorphism near every  point $(y_1,...,y_n)$  such that  $y_i$ are all distinct.

Next, take  a sufficiently small positive $\varepsilon$ and  
consider the set $Y:= \{(y_1,...,y_n)\in \mathbb{R}^n \mid -\varepsilon<y_1<...<y_n <\varepsilon \}$.
As an open subset where we prove that $g'= g_f$,  we  choose $V=U\cap \phi(Y)$. Note that  $\phi$ is injective on $Y$, since the coefficients of the polynomial 
$t^n - \sigma_1 t^{n-1} + \sigma_2 t^{n-2} - ... + (-1)^n \sigma_n$  
determine its ordered zeros $y_1 < ... < y_n$ uniquely.

 By construction, at every point $\phi(y_1,...,y_n)$ of this neighbourhood,  
$L$ has $n=\dim M$ distinct real eigenvalues $-\varepsilon<y_1<...<y_n<\varepsilon$. Therefore, in coordinates $y_1,...,y_n$ the operator $L$ is $\operatorname{diag}(y_1,...,y_n)$  and the  metric $g'$    has the following  form  (which one  immediately obtains  from \eqref{eq:LC} by  passing to the coordinates $y_1,...,y_n$ that are (ordered) eigenvalues of $L$)    
\begin{equation}
\label{gstrich}
g'= \sum_{i=1}^n \left(  f_i(y_i) \prod_{j=1; j\ne i}^n  (y_i- y_j)\right)\ddd y_i^2.    
\end{equation} 
Here $f_i$ are some functions of one variable; since the metric is real analytic, they are also real analytic.

Next,  consider the $(1,1)$--tensor $M$ given by the following formula:
\begin{equation}\label{eq:M} 
M=g^{-1} g' = g^{is}g'_{sj}.
\end{equation}
By direct calculations we see  that in coordinates $y_1,..., y_n$, the metric $g$ given by \eqref{metricg} takes the form
\begin{equation}\label{eq:gflat}   g= \sum_{i=1}^n   \left( \prod_{j=1; j\ne i}^n  (y_i- y_j)\right)dy_i^2.    
\end{equation}
Comparing formulas \eqref{eq:M} and  \eqref{eq:gflat}, we see that $M$ is given by  the diagonal matrix 
\begin{equation}
\label{eq:m} 
M= \operatorname{diag}\bigl(f_1(y_1),...,f_n(y_n)\bigr),
\end{equation}  
and therefore commutes with $L$. Since our objects are all real analytic, the operators $L$ and $M$ commute at every point of $U$.

Next, observe that  $L$ is $\mathrm{gl}$-regular on $U$  in the sense of Definition \ref{def:algregular}. This implies that $M$ is a polynomial in $L$ of degree $\le n-1$ whose coefficients  $a_i$ are functions on $U$:
\begin{equation}\label{eq:pol}
M= a_{n-1}(x)L^{n-1} +...+ a_1(x)L + a_0(x) \, \Id.
\end{equation}
Notice that the matrices $L^{n-1},...,L, \Id$ are linearly independent. Therefore the coefficients $a_i$ are unique and real analytic in $x$ on $U$.  Since the eigenvalues of $M$ at the  point $\phi(y_1,...,y_n)$ are $f_1(y_1),...,f_n(y_n)$ and the eigenvalues of $L$ at the  point $\phi(y_1,...,y_n)$ are $y_1,...,y_n$, we obtain the following formula: 
\begin{equation}\label{eq:rel}
\left(\begin{array}{c}f_1\\ \vdots  \\f_n\end{array} \right)= \left(\begin{array}{ccc} y_1^{n-1} & \cdots & 1 \\ \vdots & & \vdots \\   y_n^{n-1} & \cdots & 1 \end{array}\right)\left(\begin{array}{c}a_{n-1}\left(\phi(y_1,...,y_n)\right)\\    \vdots  \\a_{0}\left(\phi(y_1,...,y_n)\right)\end{array}\right). 
\end{equation}
Note that the $f_i(y_i)$-eigenvector of $M$ coincides with the $y_i$-eigenvector of $L$ at every point $\phi(y_1,...,y_n)$ and each of the functions $f_i$ is defined on $(-\varepsilon, \varepsilon)$.

In view of \eqref{eq:phi}, the functions $a_i(\phi(y_1,...,y_n))$ are symmetric in   $y_1,...,y_n$. This implies that all the functions $f_i$ coincide, 
i.e.,
 $f_i(y)= f(y)$ for some real analytic function $f$  defined on  $(-\varepsilon, \varepsilon)$. (Indeed,  from \eqref{eq:rel} we see that 
$$
f_i(y_i) - f_{i+1}(y_{i+1}) = \sum_{k=0}^{n-1} (y_i^{k} - y_{i+1}^{k}) a_k(\phi(y_1,\dots,y_n))
$$
and then  can take the limit as $y_i \to y_{i+1}$  (keeping all the other $y_k$'s fixed). Since the functions $a_k$ are bounded, we get $f_i(y)=f_{i+1}(y)$ for any $y\in (-\varepsilon,\varepsilon)$, as needed.)
 
Finally,   since $f_i =f$ we conclude from \eqref{eq:m} that $M=f(L)$ and \eqref{eq:M} implies that $g'=g_f=g f(L)$ on 
 $V=U\cap  \phi(Y)$. Since both $g'$ and $g_f$ are real analytic, they therefore coincide on $U$ as we claimed.  \end{proof}

\subsection{Poisson-Nijenhuis structures near differentially non-degenerate singular points}

Recall that two  Poisson structures $P$ and $\widetilde P$   are called \emph{compatible}, if their sum $P + \widetilde P$ is also a Poisson structure. If $P$ and $\widetilde P$  are  
non-degenerate and therefore come from symplectic forms $\omega =P^{-1}$ and $\widetilde \omega=\widetilde P^{-1}$, the compatibility is equivalent to the property that the {\it recursion operator} $L$ given by the relation 
\begin{equation}\label{omega12} 
\widetilde \omega(\,\cdot \,,  \cdot \,)=  \omega ( \,L \,\cdot \,,  \,\cdot  ) 
\end{equation}
is Nijenhuis.  Since $\widetilde \omega$ can be recovered from $\omega$ and $L$, we can reformulate the compatibility condition as follows  (cf. Section \ref{sect:geod}).  We will say that a symplectic structure $\omega$ and a Nijenhuis operator $L$ are {\it compatible} if
\begin{itemize}
\item[(a)]  the form $\widetilde \omega(\,\cdot \,,  \cdot \,)=  \omega ( \,L \,\cdot \,,  \,\cdot  )$ is skew-symmetric, i.e., is a differential 2-form,

\item[(b)] this form is closed, i.e., $\dd\widetilde\omega = 0$.

\end{itemize} 

In the case when $L$ is non-degenerate,  a compatible pair  $(\omega, L)$ defines a Poisson-Nijenhuis structure in the sense of  \cite{kossman, magri2}.  However, non-degeneracy of $L$ is not very essential as we can replace it with $L + c\cdot\Id$, $c\in\R$ and here we will think of compatible pairs  $(\omega, L)$ as a natural subclass of Poisson-Nijenhuis structures.

If one wants to study singularities of Poisson-Nijenhuis structures (cf. Problem 5.17 in \cite{openprob}), then it is natural to ask which Nijenhuis operators admit compatible symplectic structures and what is a simultaneous canonical form for $\omega$ and $L$ near a (possibly singular) point $p\in M$?

Notice, first of all, that condition (a) imposes natural algebraic restrictions on the algebraic type of $L$: in the Jordan decomposition of $L$ all the blocks can be partitioned into pairs of equal blocks  (i.e. of the same size and with the same eigenvalue). In particular,  the characteristic polynomial of $L$ is a full square and each eigenvalue has even multiplicity.

If  $L$ has $n$ distinct real eigenvalues,  each of multiplicity 2,  and in addition their differentials are linearly independent at a point $p\in M$, then 
there exists a local canonical coordinate system $x_1,..., x_n, p_1, ... , p_n$ in which $L$ is diagonal with $x_i$ being its $i$th eigenvalue \cite{magri2}:  
\begin{equation} 
\label{eq:GZform}
\omega = \sum_i dp_i\wedge dx_i, \ \ L= \operatorname{diag}(x_1,x_2, ... ,x_n,\, x_1,x_2,...,x_n).
\end{equation}
Equivalenly, one can say that the pair  $(\omega, L)$ is a direct sum of two-dimensional blocks $(\omega_i, L_i)$ of the form $\omega_i=dp_i\wedge dx_i$,  $L_i = x_i \cdot \Id$.  This result can be naturally generalise to the case when $L$ is semisimple and algebraically generic  but may admit multiple and  complex eigenvalues (see \cite{magri2}, \cite{turiel}, \cite{gelzak3}).

Notice that  to admit a compatible symplectic partner,  a semisimple algebraically generic Nijenhuis operator $L$ should satisfy one additional condition, namely, its non-constant eigenvalues must be all of multiplicity 2  (cf. Section \ref{sect:geod}).

If $L$ is algebraically generic but not necessarily semisimple, the description (rather non-trivial) of compatible pairs $\omega, L$  was obtained by Turiel \cite{turiel} under some additional assumptions on the differentials of $\tr L^k$, $k=1,\dots, n$.  These assumptions basically mean that each eigenvalue is either constant or its differential does not vanish.

The next natural step is to study local normal forms for $\omega$ and $L$ at singular points that are differentially non-degenerate in a natural sense\footnote{Since each eigenvalue of $L$ has multiplicity at least two, the coefficients of the characteristic polynomial of $L$ cannot be functionally independent and for this reason, Definition \ref{def:nondeg} of differentially non-degenerate singular points should be appropriately modified, see Theorem \ref{thm:gz}.}. Note that this problem  is clearly important for finite dimensional integrable systems and we do hope that it might also be important for understanding of  bihamiltonian structures in infinite dimension appearing in the theory of integrable ODEs and PDEs.

\begin{Theorem} \label{thm:gz} Let $\omega$ and $L$ be compatible  (i.e., define a Poisson-Nijenhuis structure on $M^{2n}$) and real analytic. Suppose that  at a point $p\in M^{2n}$,  all the eigenvalues of $L$ vanish  but the differentials  $\ddd\tr L$, \dots, $\ddd\tr L^n$ are linearly independent.
Then  there exists a local coordinate system $x_1,...,x_n, p_1,...,p_n$ such that $\omega=\sum_i \ddd x_i\wedge \ddd p_i$ and $L$ is given by the matrix 
\begin{equation} \label{AS1} \begin{pmatrix} A & 0_n\\ S &  A^t\end{pmatrix},\end{equation}
where 
	\begin{equation} \label{AS2}
	A =\begin{pmatrix}  -x_1&    1   &   0   &   \cdots      &  0       \\
                 -x_2&  0     &     1 &             \ddots   &   \vdots     \\  
						  \vdots & \vdots &\ddots & \ddots                &  0  \\
                -x_{n-1}& \vdots &   & \ddots &1 \\
								-x_n& 0      &\cdots &   \cdots        &            0       
\end{pmatrix}, \ \ \   
S=\begin{pmatrix}  
								  0    & \!\! -p_2    &\!\! -p_3 &   \cdots &   \!\!    -p_n       \\
								  p_2  & 0       & 0 &     \cdots &      0        \\  
                  p_3  &  0 & \vdots &             &   \vdots  \\ 
								  \vdots&  \vdots &  \vdots       &      &  \vdots     \\  
							    p_n  &   0       &0 &	 \cdots 	  &    0       
\end{pmatrix},      \end{equation} 
$0_n$ is the zero $n\times n$-matrix, and $A^t$ denotes the transposed of $A$.
\end{Theorem}

The proof of this theorem will be published separately as it is rather technical and is based on the analysis of some power series  (that is why the real analyticity is important).

\section{Conclusion and further research}

In our paper we have tried to demonstrate that Nijenhuis Geometry may become an interesting and important area of research related to and having applications in many areas of mathematics and mathematical physics.  One of our goals has been to re-direct the research agenda of the theory of Nijenhuis operators  from ``tensor analysis at generic points'' to studying singularities and global properties, as it already happened before with other areas of differential geometry.   

A very stimulating example for us in this context was Poisson Geometry.   Having its origin in classical works of Poisson, Jacobi and especially Sophus Lie (who devoted the second volume of his famous ``Theorie der Transformationsgruppen'' \cite{Lie} to the study of Poisson algebras of functions), this area of mathematics for many decades remained just an instrument for Hamiltonian and quantum mechanics.  Everybody was familiar with the definition, basic properties and Darboux canonical form and this knowledge was basically sufficient. (Aren't we in a similar situation with Nijenhuis geometry now?) The situation dramatically changed in 70-80th of last century:  ``After a long dormancy, Poisson geometry has become an active field of research during the past 30 years or so, stimulated by connections with a number of areas...'' (A.\,Weinstein, 1998 \cite{weinstein2}).  These connections revealed, in particular, the importance of {\it singularities} and {\it non-local properties} of Poisson structures (see \cite{weinstein},  \cite{Lichner}, \cite{karasev}) and that was somehow a turning point.  That is the reason for our optimism in view of further research prospects for Nijenhuis Geometry.

As pointed out in Introduction, this work is the first in a series of papers we are currently working on.  The second of them \cite{Part3} will soon be available on arXiv and is entirely devoted to linearisation of Nijenhuis operators and their relations to left-symmetric algebras in the spirit of Section \ref{sect:5}.  In particular it contains the classification of all real left-symmetric algebras (equivalently, linear Nijenhus operators) in dimension 2  and description of those of them which are non-degenerate in the sense of Definition \ref{def:nondegLSA}  in smooth and real-analytic cases (the results turned out to be slightly different).

The topic of ``Nijenhuis Geometry 3''  are $\mathrm{gl}$-regular Nijenhuis operators.    In Section \ref{sect:4}, we  gave a local description of such operators under additional  assumption of differential non-degeneracy.    The canonical form given by \eqref{eq:14} is analogous to canonical $n$-parametric versal deformations of  $\mathrm{gl}$-regular matrices in the sense of Arnol'd \cite{Arnold}.   We are going to show that `versality property' still holds for  $\mathrm{gl}$-regular Nijenhuis operators in the sense that local canonical forms, in general, can be obtained from \eqref{eq:14} by replacing the `versal' parameters $x_1,\dots, x_n$ with some smooth functions $f_1(x_1,\dots, x_n), \dots, f_n(x_1,\dots, x_n)$.  These local results are a necessary step for stydying $\mathrm{gl}$-regular Nijenhuis manifolds  (which can be understood as analogues of regular Poisson and symplectic manifolds),  i.e., such that  the Nijenhuis operator $L$ is $\mathrm{gl}$-regular at each point of $M$, but not necessarily diagonalisable (cf. Problem 5.11 in \cite{openprob}).

Part 4 will be devoted to studying nilpotent Nijenhuis operators $L$ and their normal forms.  The simplest situation when $L$ is similar to a single Jordan block has been treated in Section \ref{sect:4} and we will generalise these techniques to the general case.  The main difficulty is that in general the distributions $\Ker L^k$ are not integrable anymore.  It is expected that the classification of nilpotent Nijenhuis operators can be reduced  to classification of (flags of) distributions with additional properties implied by the relation $\mathcal N_L=0$.  The latter problem is highly non-trivial, but this reduction would be interesting on its own in any case.  We will also discuss canonical forms for Nijenhuis operators with a single eigenvalue $\lambda$ (not necessarily constant) in the case of several Jordan $\lambda$-blocks.  According to Splitting Theorem (Theorem \ref{thm1}),  completing this programme would lead to solution of the local description problem (Question 1 from Introduction) and we will do it for low dimensions.

Finally, ``Nijenhuis Geometry 5'', with a tentative subtitle {\it Nijenhuis Zoo},  will collect various examples, facts and constructions related to Nijenhuis operators with emphasis on singularities and global properties. In particular, we explain how to construct a real analytic $\R$-diagonalisable (at each point) Nijenhuis operator $L$ on a closed surface of arbitrary genus  $g\ge 0$ and why, under additional condition that each singular point of $L$ is non-degenerate, such operators exist only for $g=0$ and $1$  (sphere and torus).  We will also discuss examples of left symmetric algebras of arbitrary dimension and related constructions.
For instance, we will prove an interesting uniqueness result: if  $\det L(x) = \Pi_{i=1}^n x_i$ for a linear Nijenhuis operator $L(x)=\bigl( l^i_{jk} x_k\bigr)$, then $L=\mathrm{diag}(x_1,\dots,x_n)$. 

In conclusion, we refer to  \cite[Section 5(c)]{openprob} devoted to  open questions in Nijenhuis Geometry.


\end{document}